\documentclass[10pt,reqno]{amsart}

\usepackage{url}
\usepackage{pdfsync}

\usepackage{textcomp}
\usepackage[latin1]{inputenc}
\usepackage{amsmath,amsfonts,amssymb}
\usepackage{amscd}
\usepackage{graphicx}

\usepackage{mathrsfs}

\DeclareMathAlphabet{\mathpzc}{OT1}{pzc}{m}{it}

\usepackage[all]{xy}

\usepackage{tikz}
\usetikzlibrary{arrows,decorations.pathmorphing,decorations.pathreplacing,positioning,shapes.geometric,shapes.misc,decorations.markings,decorations.fractals,calc,patterns}

\usepackage{graphicx}

\def\BZ{\mathbb{Z}}

\def\fB{\mathfrak{B}}

\def\fD{\mathfrak{D}}

\def\fV{\mathfrak{V}}

\def\fa{\mathfrak{a}}
\def\fb{\mathfrak{b}}
\def\fc{\mathfrak{c}}
\def\fd{\mathfrak{d}}

\def\adots{\mathinner{\mkern1mu\raise1.0pt\vbox{\kern7.0pt\hbox{.}}\mkern2mu\raise4.0pt\hbox{.}\mkern2mu\raise7.0pt\hbox{.}\mkern1mu}}

\def\C{\operatorname{C}}

\def\Ext{\operatorname{Ext}}

\def\Hom{\operatorname{Hom}}

\def\nc{\operatorname{nc}}

\def\prod{\operatorname{prod}}

\def\P{\mathcal P}%
\def\C{\mathcal C}%
\def\F{\mathcal F}%
\def\G{\mathcal G}%
\def\W{\mathcal W}%

\newcommand{\Dfn}[1]{\emph{#1}} 

\numberwithin{equation}{section}

\newtheorem{Lemma}{Lemma}[section]
\newtheorem{Theorem}[Lemma]{Theorem}
\newtheorem{Proposition}[Lemma]{Proposition}
\newtheorem{Corollary}[Lemma]{Corollary}

\theoremstyle{definition}
\newtheorem{Definition}[Lemma]{Definition}

\newtheorem{Example}[Lemma]{Example}

\begin{document}

\setlength{\parindent}{0pt}
\setlength{\parskip}{2pt}

\title[Ptolemy diagrams and torsion pairs]{Ptolemy diagrams and
torsion pairs in the cluster categories of Dynkin type D}

\author{Thorsten Holm}
\address{Institut f\"{u}r Algebra, Zahlentheorie und Diskrete
  Mathematik, Fa\-kul\-t\"{a}t f\"{u}r Ma\-the\-ma\-tik und Physik, Leibniz
  Universit\"{a}t Hannover, Welfengarten 1, 30167 Hannover, Germany}
\email{holm@math.uni-hannover.de}
\urladdr{http://www.iazd.uni-hannover.de/\~{ }tholm}

\author{Peter J\o rgensen}
\address{School of Mathematics and Statistics,
Newcastle University, Newcastle upon Tyne NE1 7RU, United Kingdom}
\email{peter.jorgensen@ncl.ac.uk}
\urladdr{http://www.staff.ncl.ac.uk/peter.jorgensen}

\author{Martin Rubey}
\address{Endresstr.\ 59/14, 1230 Wien, Austria}
\email{martin.rubey@math.uni-hannover.de}
\urladdr{http://www.iazd.uni-hannover.de/~rubey}

\thanks{{\em Acknowledgement. }This work has been carried out in the framework 
  of the research priority programme SPP 1388 {\em Darstellungstheorie} of
  the Deutsche Forschungsgemeinschaft (DFG).  We gratefully acknowledge
  financial support through the grants HO 1880/4-1 and HO 1880/5-1. }

\keywords{Clique, cluster algebra, cluster category,
cluster tilting object, Dynkin type D, generating
function, species, triangulated category}

\subjclass[2010]{05A15, 05E15, 13F60, 16G10, 16G70, 18E30}

\begin{abstract} 

  We give a complete classification of torsion pairs in the cluster
  category of Dynkin type $D_n$, via a bijection to new combinatorial 
  objects called Ptolemy diagrams of type $D$. 
  For the latter we give along the way different combinatorial descriptions.  
  One of these allows us to count the number of torsion pairs in 
  the cluster category of type~$D_n$ by providing their
  generating function explicitly.

\end{abstract}

\maketitle

\section{Introduction}
\label{sec:introduction}

Torsion theory is a classic subject 
in algebra. The fundamental example of a torsion pair appears
for abelian groups with the class of torsion abelian groups and 
the class of torsion-free abelian groups as the two entries. For 
arbitrary abelian categories the concept of torsion pairs goes back 
to a paper by Dickson \cite{Dickson} from the mid 1960's. 
Since then torsion theory appeared naturally in various contexts,
in the representation theory of finite dimensional algebras most notably 
in the framework of tilting theory. In recent years the focus of several
modern developments in representation theory has been on derived categories
and related triangulated categories, e.g. stable module categories or
cluster categories. A notion of torsion pairs in triangulated categories 
has been introduced by Iyama and Yoshino \cite{IY}.  

In this paper we will study and classify combinatorially
torsion pairs in cluster categories of Dynkin type $D_n$.  

Cluster categories have been introduced by Buan, Marsh, Reineke,
Reiten and Todorov \cite{BMRRT} as a categorical 
model for Fomin and Zelevinsky's cluster algebras. Roughly speaking,
the indecomposable objects in the cluster category correspond to
the cluster variables and certain direct sums of indecomposable
objects, called cluster tilting objects, then correspond to 
the clusters in the cluster algebras. Most importantly, the 
fundamental mutation operation on clusters in cluster algebras is 
reflected by exchanging summands in cluster tilting objects in the
cluster category. This categorification approach to cluster algebras 
via cluster 
categories has been and still is highly successful in that numerous
important results on cluster algebras have been proven by using 
cluster categories. 

In representation theory, the advent of cluster categories has 
created an entirely new research area, namely cluster tilting theory; 
one of the important aspects of this new theory is that is provides a long-awaited generalization of the classic Bernstein-Bernstein-Gelfand
reflection functors and of the more general APR-tilting. 

Due to the importance of cluster categories in the theory of cluster 
algebras a lot of research goes into understanding the structure of
cluster categories. From the point of view of torsion theory a systematic
study of torsion pairs in cluster categories has only started recently.

In her thesis, Ng \cite{Ng} classified torsion pairs in the cluster 
categories of type $A_{\infty}$. These categories have been
studied in detail in \cite{HJ}; they are generated by a spherical object
and hence fit into the work of Keller, Yang and Zhou \cite{KYZ} 
where it is in particular shown that this category is uniquely determined 
up to triangulated equivalence. 

Ng's classification of torsion pairs in the cluster category of type
$A_{\infty}$ is combinatorial in the sense that she uses a combinatorial
model, namely arcs of the infinity-gon (see \cite{HJ}), to give an
explicit bijection between torsion pairs and certain configurations
of arcs.

Similar configurations, then called Ptolemy diagrams, appeared later also 
in our classification of torsion pairs in cluster categories of Dynkin type 
$A_n$ \cite{HJR-Ptolemy}. Moreover, we enumerated these torsion pairs
and gave an explicit closed formula for the number of torsion pairs
in the cluster category of type $A_n$. 

Classifications and enumerations of torsion pairs have also been 
achieved for cluster categories coming from tubes \cite{HJR-Tubes} (see 
related work of Baur, Buan and Marsh \cite{BBM} on torsion pairs in 
the abelian tube categories).

In the present paper we will provide a complete classification and
enumeration of torsion pairs in the cluster category of Dynkin type
$D_n$. The situation is more complicated than in Dynkin type $A_n$,
caused by the exceptional vertices appearing in a Dynkin diagram of 
type $D$, and hence the corresponding exceptional objects in the 
cluster category.

Our work in this paper will be based on a combinatorial model for 
the cluster category of type $D_n$ which first appeared in a paper 
by Fomin and Zelevinsky \cite{FZ-Ysystems}. There the indecomposable
objects are parametrized by pairs of rotationally symmetric arcs and
by diameters in two colours in a regular $2n$-gon. For a precise
description of this model we refer to Section \ref{sec:model} below. 

Torsion pairs in the cluster category $\mathcal{D}_n$ 
of Dynkin type $D_n$ are pairs $(\mathsf{X},\mathsf{X}^{\perp})$
of subcategories of $\mathcal{D}_n$ closed under direct sums and
direct summands which satisfy the condition 
$\mathsf{X} = ~^{\perp}(\mathsf{X}^{\perp})$; see \cite{IY}.
Here, the perpendicular
subcategories are taken with respect to $\Hom$. 
See Section \ref{sec:torsion} for more details. 

Since the subcategories appearing in a torsion pair are closed
under direct sums and direct summands they are uniquely determined 
by the set of indecomposable objects they contain. For a subcategory
$\mathsf{X}$ let $\mathcal{X}$ be the collection of arcs of the $2n$-gon
corresponding to $\mathsf{X}$ in the Fomin-Zelevinsky model. 

Our first main result gives a combinatorial characterization for
those collections $\mathcal{X}$ corresponding to a subcategory
$\mathsf{X}$ appearing as the first half of a torsion pair in
$\mathcal{D}_n$. To this end we introduce the new notion of 
a {\em Ptolemy diagram of type $D$} in Definition 
\ref{def:PtolemyD}. 

Then we go on to show the following main result in Section 
\ref{sec:Ptolemy}.

\begin{Theorem}
Let $\mathcal{X}$ be a collection of arcs of the $2n$-gon
which is invariant under rotation by 180 degrees, and let
$\mathsf{X}$ be the corresponding subcategory of the cluster
category $\mathcal{D}_n$. 
Then the following conditions are equivalent.
\begin{enumerate}
\item[{(a)}] $(\mathsf{X},\mathsf{X}^{\perp})$ is a torsion
pair in $\mathcal{D}_n$. 
\item[{(b)}] $\mathcal{X}$ is a Ptolemy diagram of type $D$.
\end{enumerate}
\end{Theorem}

As an application of this combinatorial classification we then deal 
in Section \ref{sec:counting} with the enumeration of torsion pairs in 
$\mathcal{D}_n$. We first give an alternative description of 
Ptolemy diagrams of type $D$ which is closely linked to the 
Ptolemy diagrams of Dynkin type $A$ studied
in our earlier paper \cite{HJR-Ptolemy}. 

This alternative description then allows us to work out the 
generating function for Ptolemy diagrams of type $D$, as an 
explicit expression involving the corresponding generating 
function for torsion pairs in type $A$. The main result of
Section \ref{sec:counting} can be summarized as follows 
(for unexplained notation we refer to Section \ref{sec:counting}). 

\begin{Theorem} For $n\ge 4$ let $\mathcal{D}_n$ be the cluster category
of Dynkin type $D_n$. Then the number of torsion pairs in $\mathcal{D}_n$
is given by the generating function
\begin{eqnarray*}
  \P_D(y) & := & \sum_{n\geq1} \#\{\text{Ptolemy diagrams of type~$D$ 
  of the $2n$-gon}\} y^n \\ 
& = & y\P_A^\prime(y)
\frac{1+12\P_A(y)-\P_A^2(y)-2\P_A^3(y)}{1-2\P_A(y)-\P_A^2(y)}\\
 & = &  y + 16y^2 + 82y^3 + 500y^4 + 3084y^5 + 19400y^6 +\dots\notag
\end{eqnarray*}
where $\P_A(y)$ is the generating function for Ptolemy diagrams of Dynkin
type $A$, as studied in \cite{HJR-Ptolemy}.
\end{Theorem}

The paper is organized as follows. In Section \ref{sec:model}
we recall in some detail Fomin and Zelevinsky's combinatorial model 
for a cluster category of Dynkin type $D_n$ from \cite{FZ-Ysystems}.
In Section \ref{sec:torsion}
we first briefly review the fundamentals on torsion
theory in triangulated categories, as introduced by Iyama and
Yoshino \cite{IY}. Then we apply and make explicit this concept for
the cluster categories $\mathcal{D}_n$. In particular we show
in Proposition \ref{prop:torsion}
that the following holds:  
$(\mathsf{X},\mathsf{X}^{\perp})$ is a torsion pair in $\mathcal{D}_n$
if and only if the corresponding collection $\mathcal{X}$ of arcs satisfies
$\mathcal{X} = \nc \nc \mathcal{X}$ (where $\nc \mathcal{X}$ consists of
the arcs not crossing any arc from $\mathcal{X}$). 
Section \ref{sec:Ptolemy} then constitutes one of the two main
parts of the paper. Namely, we first introduce Ptolemy diagrams of type $D$
by imposing explicit combinatorial conditions on collections of arcs 
(see Definition \ref{def:PtolemyD}) and then make a detailed 
analysis to show that these new Ptolemy diagrams of type $D$ are
precisely the collections of arcs satisfying 
$\mathcal{X} = \nc \nc \mathcal{X}$, i.e. the ones corresponding 
to torsion pairs in $\mathcal{D}_n$. 
The second main part of the paper is Section \ref{sec:counting}
in which we enumerate torsion pairs in the cluster category 
of Dynkin type $D_n$. To this end we establish an alternative 
description of Ptolemy diagrams of type $D$ via 
different types of central regions to which Ptolemy diagrams of 
type $A$ are glued.

\section{A geometric model for cluster categories of Dynkin 
type D} \label{sec:model}

In this section we will briefly recall the definition and describe
the structure
of the cluster category of Dynkin type $D_n$. Moreover, we 
recall in detail a geometric model, introduced by Fomin and
Zelevinsky \cite{FZ-Ysystems}, for this cluster category on which
we will build throughout the paper. 
\smallskip

For a quiver $Q$ without oriented cycles the cluster category 
(over a field $k$) 
has been introduced by Buan, Marsh, Reineke, Reiten and Todorov 
\cite{BMRRT} as the orbit category 
$$\mathcal{C}_Q := \mathsf{D}^b(kQ)/(\tau^{-1}\circ \Sigma)$$
where $\tau$ and $\Sigma$ are the Auslander-Reiten translation and 
the suspension on the bounded derived category of the path
algebra $kQ$. It has been shown by Keller \cite{Keller}
that $\mathcal{C}_Q$ is a triangulated category. 

Now let $Q$ be a Dynkin quiver of type $D_n$ for an integer 
$n \geq 4$.
Since the path algebras for different orientations are known to be
derived equivalent, we can assume that $Q$ has the orientation as
given in Figure \ref{fig:Dn}.
The vertices $(n-1)^{\pm}$ are called {\em exceptional}, the others
{\em non-exceptional}.
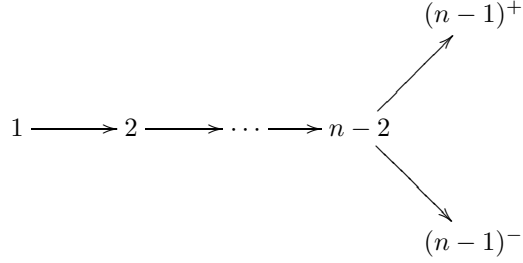
\begin{figure}
\[
  \xymatrix @+1.6pc @!0 {
             &          &               &                     & (n-1)^+  \\
    1 \ar[r] & 2 \ar[r] & \cdots \ar[r] & n-2 \ar[ur] \ar[dr] &     \\
             &          &               &                     & (n-1)^- \\
           }
\]
\caption{The Dynkin quiver $D_n$}
\label{fig:Dn}
\end{figure}
By a result of Happel 
\cite[Corollary\ 4.5(i)]{Happel}, the Auslander-Reiten quiver of 
the derived category $\mathsf{D}^b(kQ)$ is the
repetitive quiver $\BZ D_n$ shown in Figure \ref{fig:ZDn}.
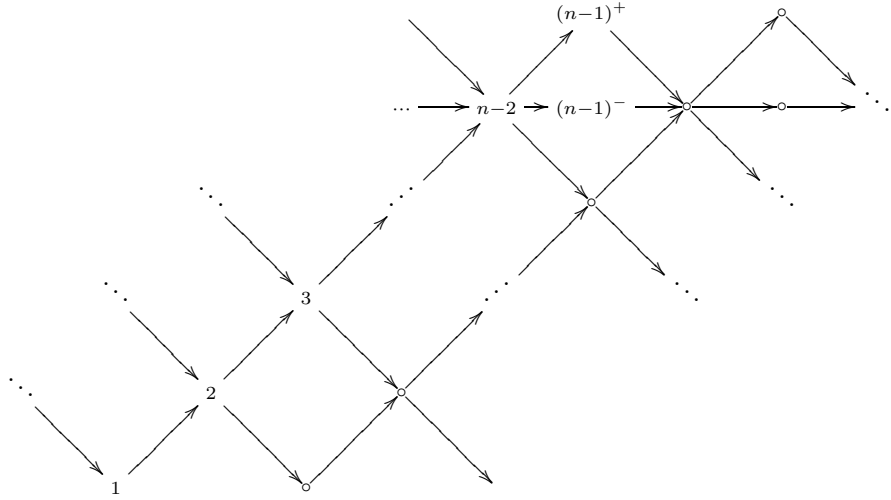
\begin{figure}
\[
  \def\objectstyle{\scriptstyle}
  \vcenter{
  \xymatrix @+1pc @!0 {
& & & & {} \ar[dr] & & (n-1)^+ \ar[dr] & & *{\circ} \ar[dr] & \\
& & & & \cdots \ar[r] & n-2 \ar[dr] \ar[ur] \ar[r] & (n-1)^- \ar[r] & *{\circ} \ar[dr] \ar[ur] \ar[r] & *{\circ} \ar[r] & \ddots \\
& & \ddots \ar[dr] & & \adots \ar[ur] & & *{\circ} \ar[dr] \ar[ur] & & \ddots & \\
& \ddots \ar[dr] & & 3 \ar[dr] \ar[ur] & & \adots \ar[ur] & & \ddots & & \\
\ddots \ar[dr] & & 2 \ar[dr] \ar[ur] & & *{\circ} \ar[dr] \ar[ur] & & & & & \\
& 1 \ar[ur] & & *{\circ} \ar[ur] & & {} & & & & \\
            }
          }
\]
\caption{The repetitive quiver $\BZ D_n$}
\label{fig:ZDn}
\end{figure}

In the sequel we denote the cluster category $\mathcal{C}_{D_n}$ 
just by $\mathcal{D}_n$. In this cluster category, objects of the derived
category are identified modulo the action of the functor $\tau^{-1}\circ \Sigma$. 
So let us describe this action explicitly. 

We will frequently 
use in this paper 
the coordinate system on $\mathbb{Z}D_n$ (first appearing in
a paper by Iyama
\cite[Section 4]{Iyama-maxorthogonal}) as given in 
Figure \ref{fig:Dn_coordinates}.

\begin{figure}
\[
  \def\objectstyle{\scriptstyle}
  \xymatrix @+1.1pc @!0 {
& & & & {} \ar[dr] & & (0,n)^+ \ar[dr] & & (1,n+1)^+ \ar[dr] & \\
& & & & \cdots \ar[r] & (0,n-1) \ar[dr] \ar[ur] \ar[r] & (0,n)^- \ar[r] & (1,n) \ar[dr] \ar[ur] \ar[r] & (1,n+1)^- \ar[r] & \ddots \\
& & \ddots \ar[dr] & & \adots \ar[ur] & & (1,n-1) \ar[dr] \ar[ur] & & \ddots & \\
& \ddots \ar[dr] & & (0,4) \ar[dr] \ar[ur] & & \adots \ar[ur] & & \ddots & & \\
\ddots \ar[dr] & & (0,3) \ar[dr] \ar[ur] & & (1,4) \ar[dr] \ar[ur] & & & & & \\
& (0,2) \ar[ur] & & (1,3) \ar[ur] & & {} & & & & \\
                        }
\]
\caption{The coordinate system in Dynkin type $D$}
\label{fig:Dn_coordinates}
\end{figure}
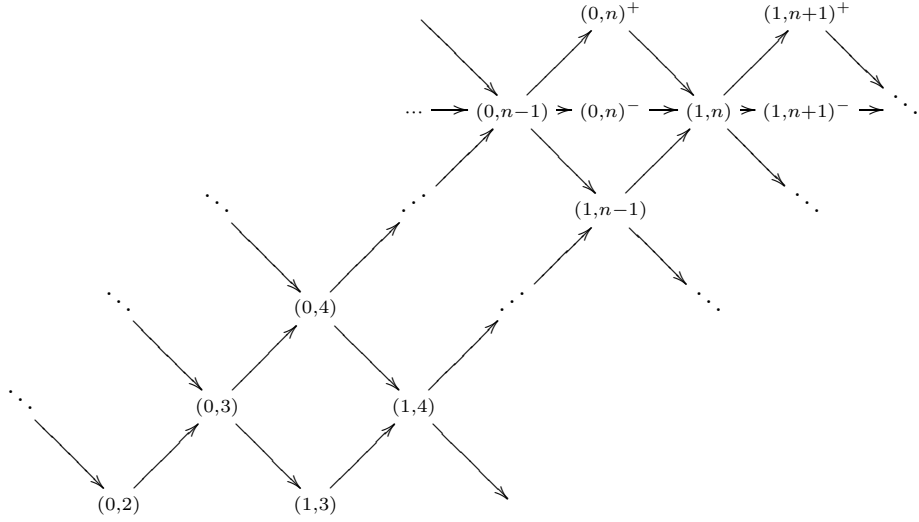

The inverse Auslander-Reiten translation $\tau^{-1}$ acts by shifting 
one unit to the right. More precisely, in the coordinate system this 
means that
$\tau^{-1}(i,j) = (i+1,j+1)$ for the non-exceptional vertices, 
and $\tau^{-1}(i,i+n)^{\pm} = (i+1,i+n+1)^{\pm}$ for the 
exceptional vertices. 

The action of the suspension functor is more subtle and depends
on the parity of $n$; it can be deduced from a paper by 
Miyachi and Yekutieli \cite[table p.\ 359]{MiyachiYekutieli}.

If $n$ is even, then the suspension $\Sigma$ acts by shifting $n-1$
units to the right. Expressed with the coordinate system we thus have
$\Sigma(i,j)=(i+n-1,j+n-1)$ for the non-exceptional vertices and
$\Sigma(i,i+n)^{\pm} = (i+n-1,i+2n-1)^{\pm}$ for the exceptional 
vertices. 

If $n$ is odd, then $\Sigma$ acts by shifting
$n-1$ units to the right and switching each pair of exceptional
vertices, i.e. 
$\Sigma(i,j)=(i+n-1,j+n-1)$ for the non-exceptional vertices, 
but for the exceptional vertices we have  
$\Sigma(i,i+n)^{\pm} = (i+n-1,i+2n-1)^{\mp}$. 

Therefore, for obtaining the Auslander-Reiten quiver of
the cluster category $\mathcal{D}_n$ one has to make the 
following identifications of vertices from the Auslander-Reiten quiver
of the derived category. For the non-exceptional vertices we have
$(\tau^{-1}\circ \Sigma)\, (i,j) = (i+n,j+n)$, i.e. indices just
have to
be taken modulo $n$. For the exceptional vertices we have
$(\tau^{-1}\circ \Sigma)\, (i,i+n)^{\pm} = (i+n,i+2n)^{\pm}$
if $n$ is even, and
$(\tau^{-1}\circ \Sigma)\, (i,i+n)^{\pm} = (i+n,i+2n)^{\mp}$
if $n$ is odd.

Accordingly, the Auslander-Reiten quiver of the cluster category
$\mathcal{D}_n$ has the shape of a cylinder of circumference
$n$; in particular $\mathcal{D}_n$ has precisely $n^2$
indecomposable objects.

There are different variations of combinatorial models 
for the cluster category of Dynkin type $D_n$. 
The first one appeared in work of Fomin and Zelevinsky
\cite{FZ-Ysystems}; it works in a regular $2n$-gon and
uses pairs of arcs which are obtained by 180 degree rotation 
and diameters in two colours (corresponding to the two 
exceptional vertices in type $D_n$). 

Later there has been a variation of this model
by Schiffler \cite{Schiffler-TypeD} using a punctured disc 
and homotopy classes of paths between vertices. This more recent 
model is often used because it fits well into the framework
of triangulations of surfaces. 

However, in this paper we will work with the Fomin-Zelevinsky 
model. For us, it has two main advantages; namely it is 
combinatorially very simple to describe and to work with 
and secondly it is analogous to the standard
combinatorial model which we
used in our earlier classification of torsion pairs in the
cluster categories of Dynkin type $A$ \cite{HJR-Ptolemy}. 

Let us now recall
the Fomin-Zelevinsky model for cluster
categories of Dynkin type $D_n$ \cite[Section 3]{FZ-Ysystems}. 

For any $n\geq 1$ we consider a regular $2n$-gon $P$
(although the Dynkin diagrams $D_n$ only appear for $n\ge 4$,
on the combinatorics side it makes sense to include the
small values of $n$). 

We label the vertices of $P$ counterclockwise by $0,1,\ldots 2n-1$ 
consecutively. In our arguments below vertices will also be numbered
by some $r\in\mathbb{N}$ which might not be in the range 
$0\le r\le 2n-1$; in this case the numbering of vertices always has 
to be taken modulo $2n$.

An {\em arc} in $P$ is a set $\{i,j\}$ of vertices of 
$P$ with $j\not\in \{i-1,i,i+1\}$, i.e. 
$i$ and $j$ are different and non-neighboring vertices. 
The arcs connecting two opposite vertices $i$ and $i+n$ are called 
{\em diameters}. We need two different copies of each of these 
diameters and denote them by $\{i,i+n\}_g$ and $\{i,i+n\}_r$, where 
$0\le i\le n-1$. The indices should indicate that these diameters are
coloured in the colours green and red, which is a convenient 
way to think about and to visualize the diameters. 
By a slight abuse of notation, we sometimes omit the indices 
and just write $\{i,i+n\}$ for diameters, to avoid cumbersome 
definitions or statements. 

Any arc in $P$ which is not a diameter is of the form $\{i,j\}$ where 
$j\in [i+2,i+n-1]$; here $[i+2,i+n-1]$ stands for the set of vertices
of the $2n$-gon $P$ which are met when going counterclockwise from 
$i+2$ to $i+n-1$ on the boundary of $P$.

Such an arc has a partner arc $\{i+n,j+n\}$ which is 
obtained from $\{i,j\}$ by a rotation by 180 degrees. We denote
the pair of arcs $\{\{i,j\},\{i+n,j+n\}\}$ by $\overline{\{i,j\}}$
throughout this paper. 

The model of Fomin and Zelevinsky \cite{FZ-Ysystems}
(see also \cite{Schiffler-TypeD}) for the cluster category 
$\mathcal{D}_n$ of Dynkin type $D_n$ builds on the following 
crucial fact parametrizing indecomposable objects in 
$\mathcal{D}_n$ by certain objects coming from the regular 
$2n$-gon $P$. Namely, the indecomposable objects in $\mathcal{D}_n$ are 
in bijection with the union of the set of pairs $\overline{\{i,j\}}$
of non-diameter arcs and the set of diameters $\{i,i+n\}_g$ 
and $\{i,i+n\}_r$ in two different colours.

The above parametrization of indecomposable objects of 
$\mathcal{D}_n$ via the $2n$-gon can be made explicit by looking
at the structure of the Auslander-Reiten quiver of the
cluster category $\mathcal{D}_n$
and the coordinate system in Figure \ref{fig:Dn_coordinates}.

For the pairs of non-diameter 
arcs $\overline{\{i,j\}}$ the corresponding 
indecomposable object has coordinates $(i,j)$; note that 
the coordinates are only determined modulo $n$ so both arcs
$\{i,j\}$ and $\{i+n,j+n\}$ in the pair $\overline{\{i,j\}}$
yield the same coordinate in the Auslander-Reiten quiver. 
The diameters $\{i,i+n\}_g$ and $\{i,i+n\}_r$ correspond to the 
exceptional vertices $(i,i+n)^+$ and $(i,i+n)^-$ 
in the Auslander-Reiten quiver. 
For specifying precisely which coloured diameter corresponds
to which of the two exceptional vertices one has to make a choice. 

We will use the following bijection between exceptional vertices
and coloured diameters; the motivation for this particular choice 
will become clear in Section \ref{sec:torsion} below; see in
particular the proof of Proposition \ref{prop:Hpm}.

We start by pairing the exceptional vertex $(0,n)^+$ with the
green diameter $\{0,n\}_g$ and $(0,n)^-$ with the
red diameter $\{0,n\}_r$. Then we continue in an alternating
manner. We assign green diameters to the exceptional
vertices $(1,n+1)^-$, $(2,n+2)^+$, $(3,n+3)^-$ etc., 
and red diameters to the exceptional vertices
$(1,n+1)^+$, $(2,n+2)^-$, $(3,n+3)^+$ etc.

It is a crucial observation that this assignment is compatible 
with the identification of vertices in the cluster category 
$\mathcal{D}_n$. In fact, if $n$ is even then 
$(i,i+n)^{\pm}$ get assigned to diameters of the same colour
as the vertex 
$(\tau^{-1}\circ \Sigma)\, (i,i+n)^{\pm} = (i+n,i+2n)^{\pm}$
obtained after shifting $n$ steps to the right. 

However, if $n$ is odd, then the functor $\tau^{-1}\circ \Sigma$ shifts
$n$ units to the right but also flips the exceptional vertices.
Therefore, also in this case any exceptional vertex
$(i,i+n)^{\pm}$ gets assigned to a diameter of the same colour
as the vertex 
$(\tau^{-1}\circ \Sigma)\, (i,i+n)^{\pm} = (i+n,i+2n)^{\mp}$
obtained after identification.

\section{Torsion theory in triangulated categories}
\label{sec:torsion}

In this section we summarize the fundamental definitions
and properties on torsion pairs in triangulated categories 
from the seminal paper by Iyama and Yoshino \cite{IY},
and then apply this abstract concept to the cluster category 
of Dynkin type $D_n$. 

A {\em torsion pair} in a triangulated category 
$\mathcal{C}$ with suspension functor $\Sigma$ is a pair 
$(\mathsf{X},\mathsf{Y})$ of full subcategories closed
under direct sums and direct summands such that
\begin{enumerate}
  \item[{(i)}] the morphism space $\operatorname{Hom}_{\mathcal{C}}(x,y)$ 
is zero for $x \in \mathsf{X}$, $y \in \mathsf{Y}$,
  \item[{(ii)}] each object $c \in \mathcal{C}$ appears in a 
  distinguished triangle $x \rightarrow
  c \rightarrow y \rightarrow \Sigma x$ with $x \in \mathsf{X}$, $y \in \mathsf{Y}$.
\end{enumerate}
This definition in particular includes t-structures, as introduced
by Beilinson, Bernstein, and Deligne \cite{BBD} (with the additional
condition $\Sigma \mathsf{X} \subseteq \mathsf{X}$) 
and the co-t-structures of Bondarko and Pauksztello
\cite{Bondarko}, \cite{Pauksztello}
(with the additional condition $\Sigma^{-1}\mathsf{X} \subseteq \mathsf{X}$). 
\smallskip

Any torsion pair $(\mathsf{X},\mathsf{Y})$
is determined by one of its entries, namely we have 
$$~\mbox{~~~\,\,\,\,\,~~~~~~~}
\mathsf{Y}=\mathsf{X}^{\perp}:=\{\, c \in \mathcal{C} \,|\, 
\operatorname{Hom}_{\mathcal{C}}(x,c) = 0 
\;\mbox{for each}\; x\in \mathsf{X} \,\},
$$
$$\mbox{and~~~}\mathsf{X}={}^{\perp}\mathsf{Y}:=\{\, c \in \mathcal{C} \,|\, 
\operatorname{Hom}_{\mathcal{C}}(c,y) = 0 
\;\mbox{for each}\; y \in \mathsf{Y} \,\}.$$ 
If the triangulated category $\mathcal{C}$ is Hom-finite over a field 
and Krull-Schmidt (conditions which are satisfied for the categories 
considered in this paper) we have the 
following characterisation, see
\cite[Prop. 2.3]{IY}.
Let $\mathsf{X}$ be a contravariantly finite
full subcategory of $\mathcal{C}$ which is closed 
under direct sums and direct summands. 
Then $(\mathsf{X},\mathsf{X}^{\perp})$ is a torsion
pair if and only if $\mathsf{X}={}^{\perp}(\mathsf{X}^{\perp})$.
\smallskip

We want to apply this general concept to the cluster category
$\mathcal{D}_n$
of Dynkin type $D_n$. First note that we can ignore the 
condition on contravariant finiteness since the cluster category 
of type $D_n$ has only finitely many indecomposable objects and hence
every subcategory is contravariantly (and also covariantly) finite. 
Moreover, any subcategory of $\mathcal{D}_n$ closed under direct
sums and direct summands is completely determined by the set of
indecomposable objects it contains. 

For understanding the perpendicular subspaces one needs to understand 
which homomorphism
spaces between indecomposable objects are non-zero. This can 
be explicitly described by certain regions in the Auslander-Reiten
quiver, as first observed by Iyama \cite[Section 4]{Iyama-maxorthogonal}. 

Consider an indecomposable object $x$ in $\mathcal{D}_n$ with
coordinates $(i,j)$ in the Auslander-Reiten quiver, cf. 
Figure \ref{fig:Dn_coordinates}.
To such an indecomposable object we consider the regions
(including boundaries) shown in Figures \ref{fig:DnH+} and 
\ref{fig:DnH-}, respectively. 
\begin{figure}
\[
  \xymatrix @+1pc @!0 {
& *{(i,i+n)} \ar@{-}[rrrr] & & & & *{(j-2,j+n-2)} \ar@{-}[dr] & \\
*{x = (i,j)} \ar@{-}[ur] \ar@{-}[ddrr] & & & {\textstyle H^+(x)} & & &*{(i+n-2,j+n-2)} \\
& & & *{(j-2,i+n)} \ar@{-}[dr] & & & \\
& & *{(j-2,j)} \ar@{-}[ur] & & *{(i+n-2,i+n)} \ar@{-}[uurr] & & \\
            }
\]
\caption{The set $H^+(x)$ in Dynkin type $D$}
\label{fig:DnH+}
\end{figure}
\begin{figure}
\[
  \xymatrix @+1pc @!0 {
& *{(i-n+2,i+2)} \ar@{-}[rrrr] & & & & *{(j-n,j)} \ar@{-}[dr] & \\
*{(i-n+2,j-n+2)} \ar@{-}[ur] \ar@{-}[ddrr] & & & {\textstyle H^-(x)} & & 
&{x = (i,j)} \\
& & & *{(j-n,i+2)} \ar@{-}[dr] & & & \\
& & *{(j-n,j-n+2)} \ar@{-}[ur] & & *{(i,i+2)} \ar@{-}[uurr] & & \\
            }
\]
\caption{The set $H^-(x)$ in Dynkin type $D$}
\label{fig:DnH-}
\end{figure}
If $x=(i,j)$ is a non-exceptional vertex, then all exceptional
vertices in these regions belong to $H^+(x)$ and $H^-(x)$,
respectively. 

If $x=(i,i+n)^{\pm}$ is an exceptional vertex then only half 
of the exceptional vertices belong to the regions $H^+(x)$
and $H^-(x)$. To be precise, if $x=(i,i+n)^{\pm}$ then 
$H^+(x)$ contains $(i+1,i+n+1)^{\mp}$, $(i+2,i+n+2)^{\pm}$, 
$(i+3,i+n+3)^{\mp}$ etc.,
but does not contain $(i+1,i+n+1)^{\pm}$, $(i+2,i+n+2)^{\mp}$, 
$(i+3,i+n+3)^{\pm}$ etc.

Similarly for the region $H^-(x)$. 

The crucial observation in \cite[Section 4]{Iyama-maxorthogonal}
(which follows from the mesh relations on the Auslander-Reiten 
quiver)  
is the following. Let $x$ be an indecomposable object in the cluster 
category $\mathcal{D}_n$. 
Then the indecomposable objects $c\in \mathcal{D}_n$ with
$\operatorname{Hom}_{\mathcal{D}_n}(x,c)\neq 0$ are precisely 
those which lie in the region $H^+(x)$. 
Similarly, the indecomposable objects $c\in \mathcal{D}_n$ with
$\operatorname{Hom}_{\mathcal{D}_n}(c,x)\neq 0$ are precisely 
those which lie in the region $H^-(x)$. 

\smallskip

Let us connect this crucial observation with the combinatorial
model of the cluster category $\mathcal{D}_n$
given by pairs $\overline{\{i,j\}}$
of non-diameter arcs and by green and red
diameters $\{i,i+n\}_g$, $\{i,i+n\}_r$ in a regular $2n$-gon. 
For this we shall need the following notion of crossings of arcs.  
For the non-diameter arcs this crossing will exactly reflect 
the geometric intuition of when two arcs cross. For the diameters 
one has to be careful with the different colours.

\begin{Definition}
\label{def:crossings}
\begin{enumerate} 
\item[{(a)}] We say that two non-diameter arcs $\{i,j\}$ and
$\{k,\ell\}$ \Dfn{cross} precisely if the elements $i,j,k,\ell$ are all
distinct and come in the order $i, k, j, \ell$
when moving around the $2n$-gon $P$ in one direction or the other
(i.e. counterclockwise or clockwise). In particular, the two
arcs in $\overline{\{i,j\}}$ do not cross. 

Similarly, in the case $j=i+n$, the above condition defines 
when a diameter $\{i,i+n\}_g$ (or $\{i,i+n\}_r$) crosses
the non-diameter arc $\{k,\ell\}$. 
\item[{(b)}] We say that two pairs $\overline{\{i,j\}}$ and
$\overline{\{k,\ell\}}$ of non-diameter arcs {\em cross} if
there exist two arcs in these two pairs which cross in the sense of
part (a). (Note that then necessarily the other two rotated arcs also
cross.)

Similarly, the diameter $\{i,i+n\}_g$ (or $\{i,i+n\}_r$) crosses
the pair $\overline{\{k,\ell\}}$ of non-diameter arcs if it 
crosses one of the arcs in $\overline{\{k,\ell\}}$.
(Note that it then necessarily crosses both arcs in $\overline{\{k,\ell\}}$.)
\item[{(c)}] Two diameters $\{i,i+n\}_g$ and $\{j,j+n\}_r$
of different colour {\em cross} if $j\not\in \{i,i+n\}$, i.e.
if they have different endpoints. But $\{i,i+n\}_g$ and 
$\{i,i+n\}_r$ do not cross. 
Moreover, any diameters of the same colour do not cross. 
\end{enumerate}
\end{Definition}

Then the vertices in the above regions $H^+(x)$ and $H^-(x)$  
of the Auslander-Reiten quiver 
can be expressed as follows in terms of arcs. 

\begin{Proposition} 
\label{prop:Hpm}
Let $x$ be a vertex in the Auslander-Reiten quiver
of $\mathcal{D}_n$ with coordinates $(i,j)$ (where in the case
of exceptional vertices this means $(i,i+n)^+$ or $(i,i+n)^-$).
\begin{enumerate}
\item[{(a)}] 
The vertices $y=(k,\ell)$ in the region $H^+(x)$ are
precisely those for which the corresponding arc $\{k,\ell\}$ of the
$2n$-gon crosses (at least) one of the arcs $\{i-1,j-1\}$ and
$\{i+n-1,j+n-1\}$. 
\item[{(b)}] 
The vertices $y=(k,\ell)$ in the region $H^-(x)$ are
precisely those for which the corresponding arc $\{k,\ell\}$ of the
$2n$-gon crosses (at least) one of the arcs $\{i+1,j+1\}$ and
$\{i+n+1,j+n+1\}$. 
\end{enumerate}
\end{Proposition}

\begin{proof}
This follows by direct inspection immediately from the definition of the
regions $H^{\pm}(x)$ and from Definition \ref{def:crossings}. 
We leave the details to the reader.

For the exceptional vertices note that the alternating membership
of the other exceptional vertices to $H^{\pm}(x)$ 
directly corresponds to the
alternating assignment of green and red colours as described at the
end of Section \ref{sec:model}. 
\end{proof}

Now we start describing torsion pairs combinatorially in terms of the
arc model. 

Recall the fact from \cite{IY} mentioned above that 
a pair $(\mathsf{X},\mathsf{X}^{\perp})$
of subcategories (closed under direct sums and direct summands)
of $\mathcal{D}_n$ is a torsion pair if and only
if $\mathsf{X}={}^{\perp}(\mathsf{X}^{\perp})$. 

The subcategory $\mathsf{X}$ has to be closed under direct sums 
and direct summands hence is determined by the set of indecomposable
objects of $\mathcal{D}_n$ it contains. Let $\mathcal{X}$ be the
collection of non-diameter arcs and coloured diameters of the $2n$-gon
corresponding to the indecomposable objects of $\mathsf{X}$; note that
the non-diameter arcs in this collection $\mathcal{X}$ come in pairs
obtained by 180 degree rotation, i.e. the collection $\mathcal{X}$ 
is invariant under rotation by 180 degrees. 

The following definition will be crucial. 

\begin{Definition} For $n\ge 1$ let $P$ be a regular $2n$-gon. 
If $\mathcal{X}$ is a set of arcs in $P$, then we set 
\[
  \nc \mathcal{X} = \{\, \mbox{$\alpha=\{i,j\}$ is an arc in $P$}
  \,\mid\, \mbox{$\alpha$ crosses no arc in $\mathcal{X}$} \,\}.
\]
(Note that $\alpha$ can be a diameter $\{i,i+n\}_g$ or $\{i,i+n\}_r$
here, we avoid the indices for simplicity.)
\end{Definition} 

\begin{Example} \label{ex:nc}
Let us consider two examples for $n=5$, i.e. collections of arcs
of a regular 10-gon which are invariant under 180 degree rotation. 
For better visibility we draw the red diameters in a wavelike 
form and the green ones as straight lines. 
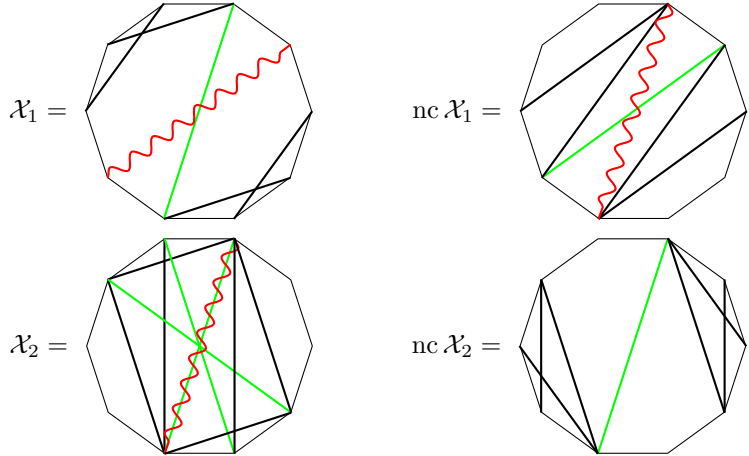
\[ 
  \begin{tikzpicture}[auto]
    \node[name=s, shape=regular polygon, regular polygon sides=10, minimum size=3cm, draw] {}; 
    \draw[shift=(s.corner 4)]  node[left]  {$\mathcal{X}_1=~~$};
    \draw[thick,green] (s.corner 1) to (s.corner 6);
    \draw[thick] (s.corner 2) to (s.corner 4);
    \draw[thick] (s.corner 1) to (s.corner 3);
    \draw[thick] (s.corner 6) to (s.corner 8);
    \draw[thick] (s.corner 7) to (s.corner 9);
    \draw[thick,decorate,decoration=snake,red] (s.corner 5) to (s.corner 10);
  \end{tikzpicture} \hskip1.2cm
  \begin{tikzpicture}[auto]
    \node[name=s, shape=regular polygon, regular polygon sides=10, minimum size=3cm, draw] {}; 
    \draw[shift=(s.corner 4)]  node[left]  
    {$\nc \mathcal{X}_1=~~$};
    \draw[thick,green] (s.corner 5) to (s.corner 10);
    \draw[thick] (s.corner 1) to (s.corner 4);
    \draw[thick] (s.corner 1) to (s.corner 5);
    \draw[thick] (s.corner 6) to (s.corner 9);
    \draw[thick] (s.corner 6) to (s.corner 10);
    \draw[thick,decorate,decoration=snake,red] (s.corner 1) to (s.corner 6);
  \end{tikzpicture} 
\]
\[ 
  \begin{tikzpicture}[auto]
    \node[name=s, shape=regular polygon, regular polygon sides=10, minimum size=3cm, draw] {}; 
    \draw[shift=(s.corner 4)]  node[left]  {$\mathcal{X}_2=~~$};
    \draw[thick] (s.corner 1) to (s.corner 3);
    \draw[thick] (s.corner 2) to (s.corner 6);
    \draw[thick] (s.corner 3) to (s.corner 6);
    \draw[thick,green] (s.corner 1) to (s.corner 6);
    \draw[thick,green] (s.corner 2) to (s.corner 7);
    \draw[thick,green] (s.corner 3) to (s.corner 8);
    \draw[thick,decorate,decoration=snake,red] (s.corner 1) to (s.corner 6);
    \draw[thick] (s.corner 6) to (s.corner 8);
    \draw[thick] (s.corner 1) to (s.corner 7);
    \draw[thick] (s.corner 1) to (s.corner 8);
  \end{tikzpicture} \hskip1.2cm
  \begin{tikzpicture}[auto]
    \node[name=s, shape=regular polygon, regular polygon sides=10, minimum size=3cm, draw] {}; 
    \draw[shift=(s.corner 4)]  node[left]  
    {$\nc \mathcal{X}_2=~~$};
    \draw[thick,green] (s.corner 1) to (s.corner 6);
    \draw[thick] (s.corner 3) to (s.corner 6);
    \draw[thick] (s.corner 3) to (s.corner 5);
    \draw[thick] (s.corner 4) to (s.corner 6);
    \draw[thick] (s.corner 1) to (s.corner 8);
    \draw[thick] (s.corner 1) to (s.corner 9);
    \draw[thick] (s.corner 8) to (s.corner 10);
  \end{tikzpicture} 
\]
\end{Example}

Then torsion pairs in $\mathcal{D}_n$ can be characterized 
combinatorially as follows. 

\begin{Proposition}
\label{prop:torsion}
Let $\mathsf{X}$ be a subcategory of the cluster category $\mathcal{D}_n$,
closed under direct sums and direct summands, and let 
$\mathcal{X}$ be the corresponding collection of arcs of the regular
$2n$-gon. 

Then $(\mathsf{X},\mathsf{X}^{\perp})$ is a torsion pair in 
$\mathcal{D}_n$ if and only if $\mathcal{X} = \nc \nc \mathcal{X}$. 
\end{Proposition}

\begin{proof}
Let $x$ and $y$ be indecomposable objects in 
$\mathcal{D}_n$, given by their coordinates $(i,j)$ and $(k,\ell)$
in the Auslander-Reiten quiver. Then we have that 
$$
\Ext^1_{\mathcal{D}_n}(x,y)=
\Hom_{\mathcal{D}_n}(x,\Sigma y)= 
\Hom_{\mathcal{D}_n}(\Sigma^{-1}x,y)\neq0
$$
if and only if the corresponding arcs 
$\{i,j\}$ and $\{k,\ell\}$
cross. In fact, by Iyama's observation from \cite{Iyama-maxorthogonal}, 
$(k,\ell)$ has to be in the region $H^+(\Sigma^{-1}x)$ which equals
$H^+(\tau^{-1}x)$ since $\tau=\Sigma$ in the cluster category 
$\mathcal{D}_n$. But $\tau^{-1}x$ has coordinates $(i+1,j+1)$
and then by Proposition \ref{prop:Hpm}\,(a)
the arc corresponding to $y$ crosses the arc corresponding to $x$. 

For the perpendicular subcategory we therefore have
\begin{eqnarray*} 
\mathsf{X}^{\perp} & = & \{\, c \in \mathcal{D}_n \,|\, 
\operatorname{Hom}_{\mathcal{D}_n}(x,c) = 0 
\;\mbox{for each}\; x\in \mathsf{X} \,\}\\
& = & \{\, c \in \mathcal{D}_n \,|\, 
\operatorname{Ext}^1_{\mathcal{D}_n}(x,\Sigma^{-1}c) = 0 
\;\mbox{for each}\; x\in \mathsf{X} \,\} \\ 
& = & \{\, c \in \mathcal{D}_n \,|\, 
\operatorname{Ext}^1_{\mathcal{D}_n}(\Sigma x,c) = 0 
\;\mbox{for each}\; x\in \mathsf{X} \,\},
\end{eqnarray*}
which corresponds to the set $\nc \Sigma\, \mathcal{X}$ of arcs,
by definition of the $\nc$ operator.

Similarly, the left perpendicular subcategory $~^{\perp}X$ corresponds
to $\nc \Sigma^{-1}\mathcal{X}$. 

Thus, $(\mathsf{X},\mathsf{X}^{\perp})$ is a torsion pair in 
$\mathcal{D}_n$ if and only if 
$\mathcal{X} = \nc \Sigma^{-1}(\nc \Sigma \mathcal{X})
= \nc \nc \mathcal{X}.$
For the last equation note that $\Sigma^{\pm}$, when interpreted in the
arc model, just induces a rotation and hence commutes with the
$\nc$ operator. 
\end{proof}

\section{Ptolemy diagrams of type $D_n$}
\label{sec:Ptolemy}

We have seen in Proposition \ref{prop:torsion} that
torsion pairs in the cluster category $\mathcal{D}_n$ 
can be characterized via their corresponding 
sets of arcs by the condition $\mathcal{X} = \nc \nc \mathcal{X}$. 

The aim of this section is to characterize combinatorially 
those collections $\mathcal{X}$ of arcs of the $2n$-gon 
which are invariant under 180 degree rotation (i.e. correspond to 
a collection of indecomposable objects of $\mathcal{D}_n$) and
satisfy $\mathcal{X} = \nc \nc \mathcal{X}$. It will turn out below,
in the main result Theorem \ref{thm:ncnc_vs_Ptolemy},
that the following properties are the crucial ones. 

The notion {\em Ptolemy diagram} is used because of the
analogy to the Ptolemy diagrams in Dynkin type $A$ \cite{HJR-Ptolemy}
whose
visualization is very reminiscent of Ptolemy's theorem about the relation 
beteween the lengths of the sides and the diagonals in a cyclic 
quadrilateral, see Figure \ref{fig:Ptolemy-A-Condition}.

\begin{Definition}
  \label{def:PtolemyD}
\begin{enumerate}
\item[{(a)}] 
  Let $\mathcal{X}$ be a collection of arcs of the $2n$-gon, $n> 1$, 
  which is invariant
  under rotation of $180$ degrees.  Then $\mathcal{X}$ is called a
  \Dfn{Ptolemy diagram of type~$D_n$} if it satisfies the following
  conditions.  Let $\alpha = \{i,j \}$ and $\beta = \{k,\ell \}$ 
  be \emph{crossing} arcs in $\mathcal{X}$ (in the sense of
  Definition \ref{def:crossings}).
  \begin{enumerate}
  \item[{(Pt1)}] If $\alpha$ and $\beta$ are not diameters, then 
  those of $\{ 
    i,k\}$, $\{i,\ell \}$, $\{
    j,k \}$, $\{ j,\ell \}$ which are
    arcs in $P$ are also in $\mathcal{X}$. 
    In particular, if two of the vertices $i,j,k,\ell$ are 
    opposite vertices (i.e. one of $k$ and $\ell$ is equal to 
    $i+n$ or $j+n$), then 
    both the green and the red diameter connecting them are also
    in $\mathcal{X}$.
  \item[{(Pt2)}] If both $\alpha$ and $\beta$ are diameters 
  (necessarily of different colour by Definition \ref{def:crossings}\,(c))   
  then those of $\{ i,k\}$, $\{
    i,k+n \}$, $\{ i+n,k \}$, $\{
    i+n,k+n \}$ which are arcs of $P$ are also in $\mathcal{X}$.
\item[{(Pt3)}] If $\alpha$ is a diameter while $\beta$ is not a diameter, 
  then those of $\{
    i,k\}$, $\{ i,\ell \}$, $\{
    j,k \}$, $\{ j,\ell \}$ which are
    arcs and do not cross the arc $\{k+n,\ell +n \}$ are also in 
    $\mathcal{X}$.
    Additionally, the diameters $\{ k, k+n \}$ and
    $\{ \ell, \ell +n \}$ of the same colour as $\alpha$ are
    also in $\mathcal{X}$.
  \end{enumerate}
\item[{(b)}] For the $2$-gon, there is precisely 
one Ptolemy diagram of type $D_1$ 
containing only the two diameters.
\item[{(c)}]   
A collection $\mathcal{X}$ of arcs is called a {\em Ptolemy diagram
of type $D$} if it is a Ptolemy diagram of type $D_n$ for some
$n\ge 1$. 
\end{enumerate}
\end{Definition}

These conditions are illustrated in
Figure \ref{tab:Ptolemy-D-Condition}, where dashed lines
indicate non-diameter arcs and diameters forced by the crossing of 
$\alpha=\{i,j\}$ and~$\beta=\{k,\ell\}$. Note that in 
Definition \ref{def:PtolemyD} the collection of arcs $\mathcal{X}$
is supposed to be invariant under rotation by 180 degrees. Conditions 
(Pt1) and (Pt3) are only formulated for the one
crossing of $\alpha$ and $\beta$, but the rotated arcs are also
crossing arcs in $\mathcal{X}$. Therefore, (Pt1) and (Pt3) also 
guarantee that the rotated arcs appearing in the pictures in
Figure \ref{tab:Ptolemy-D-Condition} 
are also in $\mathcal{X}$, although they are
not explicitly mentioned in Definition \ref{def:PtolemyD}. 

Note that in Example \ref{ex:nc}, the collection $\mathcal{X}_1$
is not a Ptolemy diagram (conditions (Pt1) and (Pt2) are violated),
whereas the collection $\mathcal{X}_2$ is a Ptolemy diagram.

\begin{figure} 
  \centering
  \begin{enumerate}
  \item[{(Pt1)}] 
  The first Ptolemy condition in Dynkin type
    $D$:
\[
  \begin{tikzpicture}[auto]
    \node[name=s, shape=regular polygon, regular polygon sides=22, minimum size=4cm, draw] {}; 
    \draw[shift=(s.corner 4)]  node[left]  {$i$};
    \draw[shift=(s.corner 20)] node[right] {$j$};
    \draw[shift=(s.corner 7)]  node[left]  {$k$};
    \draw[shift=(s.corner 22)] node[above] {$\ell$};
   \draw[thick] (s.corner 4) to node[very near start,below=13pt]{$~$}  (s.corner 20);
    \draw[thick] (s.corner 7) to node[near start, right=10pt] {$~$} (s.corner 22);
    \draw[thick,dashed] (s.corner 4) to (s.corner 7);
    \draw[thick,dashed] (s.corner 4) to (s.corner 22);
    \draw[thick,dashed] (s.corner 20) to (s.corner 7);
    \draw[thick,dashed] (s.corner 20) to (s.corner 22);
    \draw[shift=(s.corner 15)] node[right] {$i+n$};
    \draw[shift=(s.corner 9)] node[left]  {$j+n$};
    \draw[shift=(s.corner 18)] node[right] {$k+n$};
    \draw[shift=(s.corner 11)] node[below] {$\ell +n$};
    \draw[thick] (s.corner 15) to (s.corner 9);
    \draw[thick] (s.corner 18) to (s.corner 11);
    \draw[thick,dashed] (s.corner 15) to (s.corner 18);
    \draw[thick,dashed] (s.corner 15) to (s.corner 11);
    \draw[thick,dashed] (s.corner 9)  to (s.corner 18);
    \draw[thick,dashed] (s.corner 9)  to (s.corner 11);
  \end{tikzpicture} 
  \begin{tikzpicture}[auto]
    \node[name=s, shape=regular polygon, regular polygon sides=22, minimum size=4cm, draw] {}; 
    \draw[shift=(s.corner 4)]  node[left]  {$i$};
    \draw[shift=(s.corner 19)] node[right] {$j=k+n$};
    \draw[shift=(s.corner 8)]  node[left]  {$k=j+n$};
    \draw[shift=(s.corner 22)] node[above] {$\ell$};
    \draw[thick] (s.corner 4) to node[very near start,below=15pt] {$~$} (s.corner 19);
    \draw[thick] (s.corner 8) to node[above=-5pt] {$~$} (s.corner 22);
    \draw[thick,dashed] (s.corner 4) to (s.corner 8);
    \draw[thick,dashed] (s.corner 4) to (s.corner 22);
    \draw[thick,dashed,green] (s.corner 8) to (s.corner 19);
    \draw[thick,dashed,decorate,decoration=snake,red] (s.corner 19) to (s.corner 8);
    \draw[thick,dashed] (s.corner 19) to (s.corner 22);
    \draw[shift=(s.corner 15)] node[right] {$i+n$};
    \draw[shift=(s.corner 11)] node[below] {$\ell +n$};
    \draw[thick] (s.corner 15) to (s.corner 8);
    \draw[thick] (s.corner 19) to (s.corner 11);
    \draw[thick,dashed] (s.corner 15) to (s.corner 19);
    \draw[thick,dashed] (s.corner 15) to (s.corner 11);
    \draw[thick,dashed] (s.corner 8)  to (s.corner 11);
  \end{tikzpicture} 
\]

\item[{(Pt2)}] 
The second Ptolemy condition in Dynkin type
  $D$:
\[
  \begin{tikzpicture}[auto]
    \node[name=s, shape=regular polygon, regular polygon sides=22, minimum size=4cm, draw] {}; 
    \draw[shift=(s.corner 4)]  node[left]  {$i$};
    \draw[shift=(s.corner 15)] node[right] {$i+n$};
    \draw[shift=(s.corner 11)] node[below] {$k$};
    \draw[shift=(s.corner 22)] node[above] {$k+n$};
    \draw[thick,decorate,decoration=snake, red] (s.corner 4) to (s.corner 15);
    \draw (s.corner 11) to node[near start] {$~$} (s.corner 22);
    \draw[thick, green] (s.corner 11) to (s.corner 22);
    \draw[thick,dashed] (s.corner 4) to (s.corner 11);
    \draw[thick,dashed] (s.corner 4) to (s.corner 22);
    \draw[thick,dashed] (s.corner 15) to (s.corner 11);
    \draw[thick,dashed] (s.corner 15) to (s.corner 22);
  \end{tikzpicture} 
\]

\item[{(Pt3)}] 
The third Ptolemy condition in Dynkin type
  $D$:
\[
  \begin{tikzpicture}[auto]
    \node[name=s, shape=regular polygon, regular polygon sides=22, minimum size=4cm, draw] {}; 
    \draw[shift=(s.corner 4)]  node[above]  {$i$};
    \draw[shift=(s.corner 15)] node[right] {$i+n$};
    \draw[shift=(s.corner 6)]  node[left] {$k$};
    \draw[shift=(s.corner 22)] node[above] {$\ell$};
    \draw[shift=(s.corner 17)]  node[right] {$k+n$};
    \draw[shift=(s.corner 11)]  node[below] {$\ell +n$};
    \draw (s.corner 4) to node[near end, left=1pt] {$~$} (s.corner 15);
    \draw[thick, green] (s.corner 4) to (s.corner 15);
    \draw[thick] (s.corner 6) to node[near end, below=0pt] {$~$} (s.corner 22);
    \draw[thick] (s.corner 17) to (s.corner 11);
    \draw[thick,dashed] (s.corner 4) to (s.corner 6);
    \draw[thick,dashed] (s.corner 4) to (s.corner 22);
    \draw[thick,dashed] (s.corner 15) to (s.corner 11);
    \draw[thick,dashed] (s.corner 15) to (s.corner 17);
    \draw[thick,dashed, green] (s.corner 6) to (s.corner 17);
    \draw[thick,dashed, green] (s.corner 22) to (s.corner 11);
  \end{tikzpicture} 
\]
\end{enumerate}
  \caption{The Ptolemy conditions in Dynkin type~$D$.}
  \label{tab:Ptolemy-D-Condition}
\end{figure}

We are now in the position to state the main result of this section.

\begin{Theorem}
\label{thm:ncnc_vs_Ptolemy}
Let $\mathcal{X}$ be a collection of arcs of the $2n$-gon, $n\ge 1$, 
which is invariant under rotation of $180$ degrees. Then the 
following conditions are equivalent:
\begin{itemize}
\item[{(a)}] $\mathcal{X} = \nc \nc \mathcal{X}$.
\item[{(b)}] $\mathcal{X}$ is a Ptolemy diagram of type $D$.
\end{itemize}
\end{Theorem}

Before embarking on the proof of Theorem \ref{thm:ncnc_vs_Ptolemy}
let us draw a few consequences. Note that these are not 
obvious from the combinatorial definition of Ptolemy diagrams 
of type $D$ in Definition \ref{def:PtolemyD}.

\begin{Corollary}
\begin{enumerate}
\item[{(a)}] If $\mathcal{X}$ is a Ptolemy diagram of type $D$,
then $\nc \mathcal{X}$ is also a Ptolemy diagram of type $D$. 
\item[{(b)}] For any $n\ge 1$, the operator $\nc$ induces a 
bijection on the Ptolemy diagrams of type $D$ of the $2n$-gon. 
\end{enumerate}
\end{Corollary}

\begin{proof}
(a) By assumption on $\mathcal{X}$ and
Theorem \ref{thm:ncnc_vs_Ptolemy} we have that 
$\mathcal{X} = \nc\nc\mathcal{X}$. But then
$\nc\mathcal{X} = \nc (\nc\nc\mathcal{X}) = \nc\nc (\nc\mathcal{X})$,
and hence $\nc\mathcal{X}$ is again a Ptolemy diagram of type $D$
by Theorem \ref{thm:ncnc_vs_Ptolemy}.

(b) The map induced by the operator $\nc$ on Ptolemy diagrams of type
$D$ is injective by Theorem \ref{thm:ncnc_vs_Ptolemy}.
In fact if $\mathcal{X}$ and $\mathcal{Y}$ are Ptolemy diagrams 
of type $D_n$ and if $\nc\mathcal{X} = \nc\mathcal{Y}$ 
then applying $\nc$ again we get $\nc\nc\mathcal{X}=\nc\nc\mathcal{Y}$
from which $\mathcal{X}=\mathcal{Y}$ follows by using 
Theorem \ref{thm:ncnc_vs_Ptolemy}. Clearly there are only finitely
many Ptolemy diagrams of type $D_n$, so the map
induced by $\nc$ is also surjective. 
\end{proof}

As a preparation for the proof of Theorem \ref{thm:ncnc_vs_Ptolemy}
we shall first state and prove
a few useful lemmas. 
\smallskip

The following notation will be useful in the sequel: 
for any vertices $i$ and 
$j$ of the $2n$-gon, we denote by $[i,j]$ the set of all vertices 
of the $2n$-gon which are met when going counterclockwise
from $i$ to $j$ on the boundary of $P$ (including $i$ and $j$ themselves).
Recall that our numbering of vertices of $P$ was also counterclockwise 
so that $[i,j]$ can be thought of as the interval between $i$ and $j$. 
Also note that the order now matters, $[i,j]$ and $[j,i]$ are different 
sets of vertices.

\begin{Lemma} \label{lem:non-diameter}
Let $\mathcal{X}$ be a Ptolemy diagram of type $D_n$,
$n\ge 1$.
Suppose that we had a pair $\overline{\{i,j\}}$ of non-diameter 
arcs of $P$ which is in $\nc \nc \mathcal{X}$ but not in $\mathcal{X}$.
\begin{enumerate}
\item[{(a)}] If the diameters $\{i,i+n\}_g$ and $\{i,i+n\}_r$ 
are not in $\mathcal{X}$ then there exists an arc $\{i,t\}$ in
$\mathcal{X}$ where $t\in [j+1,i+n-1]$ if $j\in [i,i+n]$
and where $t\in [i+n+1,j-1]$ if $j\in [i+n,i]$.
\item[{(b)}] 
For one of the colours, the diameters attached to $i$ and $j$ 
of the same colour are 
in $\mathcal{X}$ (i.e. $\{i,i+n\}_g$ and $\{j,j+n\}_g$ are 
in $\mathcal{X}$
or $\{i,i+n\}_r$ and $\{j,j+n\}_r$ are in $\mathcal{X}$). 
\end{enumerate}
\end{Lemma}

\begin{proof}
We will only consider the case that $j\in [i,i+n]$, i.e.
that the consecutive counterclockwise order of the vertices is
$i,j,i+n,j+n$. The other case $j\in [i+n,i]$ in which the vertices
appear in counterclockwise order as $i,j+n,i+n,j$ is completely
symmetric.   
\smallskip

(a) Consider the arc $\{i-1,i+1\}$. It crosses 
$\{i,j\}\in \nc\nc\mathcal{X}$, hence $\{i-1,i+1\}$ must be crossed 
by an element from $\mathcal{X}$. By assumption, the diameters 
attached to $i$ are not in $\mathcal{X}$. So there exists
a non-diameter arc $\{i,t\}\in \mathcal{X}$ where $t\neq i+n$ and
$t\not\in \{i-1,i,i+1\}$ (otherwise $\{i,t\}$ was not an arc).
Moreover, by assumption the arc
$\{i,j\}\not\in \mathcal{X}$, thus also $t\neq j$. 

If $t\in [j+1,i+n-1]$ then the claim in (a) is shown and we are done. 

So we are left with two possibilities, namely $t\in [i+2,j-1]$
or $t\in [i+n+1,i-2]$. 

Case 1: Let $t\in [i+2,j-1]$. W.l.o.g. we can suppose
that $\{i,t\}$ has maximal length among the arcs $\{i,u\}\in \mathcal{X}$
with $u\in [i+2,j-1]$ (i.e. the arcs $\{i,u\}$ with $u\in [t+1,j-1]$
are not in $\mathcal{X}$). 

Now consider the arc $\{i-1,t\}$. It crosses 
$\{i,j\}\in \nc\nc\mathcal{X}$, hence $\{i-1,t\}$ must be 
crossed by an element from $\mathcal{X}$.

But $\{i-1,t\}$ can not be crossed by a diameter from $\mathcal{X}$.
In fact, the diameters $\{i,i+n\}_g$ and $\{i,i+n\}_r$ are not in
$\mathcal{X}$ by assumption; if a diameter $\{t',t'+n\}_g$ (or
$\{t',t'+n\}_r$) with $t'\in [i+1,t-1]$ was in $\mathcal{X}$
then condition (Pt3) implied that $\{i,i+n\}_g\in \mathcal{X}$ 
(or $\{i,i+n\}_r\in \mathcal{X}$), contradicting the assumption. 

Therefore $\{i-1,t\}$ must be crossed by a non-diameter arc 
$\{u,s\}\in \mathcal{X}$ where $u\in [i,t-1]$ (and $s\in [t+1,i-2]$). 

If $s\in [t+1,j]$ then condition (Pt1) provides an arc $\{i,s\}$
in $\mathcal{X}$
which contradicts the maximality of $t$ (if $s\neq j$)
or the assumption $\{i,j\}\not\in \mathcal{X}$ (if $s=j$).

If $s\in [j+1,i+n-1]$ then condition (Pt1) implies the existence of
an arc as claimed in (a) and we are done. 

If $s=i+n$ then condition (Pt1) implies that the diameters 
attached to $i$ are in $\mathcal{X}$, contradicting the assumption. 

So it remains to deal with the case $s\in [i+n+1,i-2]$. 
By condition (Pt1) then also the arc $\{i,s\}\in \mathcal{X}$. 
W.l.o.g. choose $s\in [i+n+1,i-2]$ so that $\{i,s\}\in \mathcal{X}$
but $\{i,s'\}\not\in \mathcal{X}$ for all $s'\in [i+n+1,s-1]$. 
See Figure \ref{fig:picture1} for an illustration.

\begin{figure} 
$$  \begin{tikzpicture}[auto]
    \node[name=s, shape=regular polygon, regular polygon sides=22, minimum size=4cm, draw] {}; 
    \draw[shift=(s.corner 5)]  node[left]  {$t$};
    \draw[shift=(s.corner 14)] node[below] {$s$};
    \draw[shift=(s.corner 7)]  node[left] {$j$};
    \draw[shift=(s.corner 22)] node[right] {$i$};
    \draw[shift=(s.corner 18)]  node[right] {$j+n$};
    \draw[shift=(s.corner 11)]  node[below] {$i +n$};
    \draw[thick,dotted] (s.corner 5) to (s.corner 14);
    \draw[thick] (s.corner 5) to (s.corner 22);
    \draw[thick] (s.corner 14) to (s.corner 22);
    \draw[thick,dotted] (s.corner 7) to (s.corner 22);
    \draw[thick,dotted] (s.corner 11) to (s.corner 18);
  \end{tikzpicture} 
$$
\caption{~}\label{fig:picture1}
\end{figure}

Now consider the arc (possibly a diameter) $\{t,s\}$. It crosses
$\{i,j\}\in \nc\nc\mathcal{X}$, so it must be crossed by an
element in $\mathcal{X}$. 

If $\{t,s\}$ is crossed by an arc in $\mathcal{X}$
attached at $i$, then the other
endpoint of this arc must be in $[j+1,i+n-1]$ and we are done; 
this follows from the choice of $t$ and $s$ and from 
the assumptions that the diameters attached at $i$ are not in 
$\mathcal{X}$ and that $\{i,j\}\not\in \mathcal{X}$. 

So we can suppose that $\{t,s\}$ is crossed by an arc (possibly a
diameter) in $\mathcal{X}$ which also crosses one of 
$\{i,t\}\in\mathcal{X}$ and $\{i,s\}\in \mathcal{X}$. 

Now by arguments analogous to the above ones 
one uses the choice of $t$ and $s$ (as endpoints of arcs of
maximal length) to conclude from the Ptolemy conditions 
(Pt1) and (Pt3) that there must be an arc $\{i,v\}$ with 
$v\in [j+1,i+n-1]$, as claimed. 
\medskip

Case 2: Let $t\in [i+n+1,i-2]$. 

This case is completely analogous (by symmetry) to Case 1; just
interchange the roles of $t$ and $s$. 
\medskip

This completes the proof of part (a). 
\medskip

(b) Again it suffices by smmetry to deal with the case where
$j\in [i,i+n]$, i.e. where $i,j,i+n,j+n$ is the counterclockwise
order of these vertices. 

Suppose first that all four diameters attached at $i$ and $j$
were not in $\mathcal{X}$. Then we can apply part (a) to both $i$ and 
$j$. This gives arcs $\{i,t\}\in \mathcal{X}$ with $t\in[j+1,i+n-1]$
and $\{j,s\}\in \mathcal{X}$ with $s\in [j+n+1,i-1]$. These two arcs
clearly cross and condition (Pt1) implied that $\{i,j\}\in \mathcal{X}$,
contradicting the assumption. 

Secondly, suppose that for one of $i$ and $j$, say for $i$, 
a diameter attached at $i$ is in $\mathcal{X}$, but for the other 
vertex $j$, no diameter attached at $j$ is in $\mathcal{X}$. 
Then part (a), applied to $j$ yields an arc $\{j,s\}\in \mathcal{X}$ 
with $s\in [j+n+1,i-1]$. This arc crosses the diameter attached at $i$
which is supposed to be in $\mathcal{X}$. But then condition 
(Pt2) implies that $\{i,j\}\in \mathcal{X}$, a contradiction.   

Finally, if for both $i$ and $j$ at least one diameter attached
to each of them is in $\mathcal{X}$ then these two diameters
must have the same colour, as claimed in (b). In fact, if the 
two diameters had different colours they would cross and condition
(Pt2) implied that $\{i,j\}\in \mathcal{X}$, contradicting the
assumption.  
\end{proof}

\begin{Lemma} \label{lem:diameter}
Let $\mathcal{X}$ be a Ptolemy diagram of type $D_n$, $n\ge 1$.
Suppose that there was a diameter, say $\{i,i+n\}_r$, which is in 
$\nc \nc \mathcal{X}$ but not in $\mathcal{X}$.

Then the diameter $\{i,i+n\}_g$ of the other colour must be in
$\mathcal{X}$. 
\end{Lemma}

\begin{proof}
We consider the non-diameter arc $\{i-1,i+1\}$. It crosses 
the diameter $\{i,i+n\}_r$ which is in $\nc\nc\mathcal{X}$ 
by assumption. Hence $\{i-1,i+1\}$ must be crossed by an element
from $\mathcal{X}$.  

If it is crossed by a diameter in $\mathcal{X}$
then the only possibility is that $\{i,i+n\}_g\in \mathcal{X}$ 
(since $\{i,i+n\}_r\not\in \mathcal{X}$ by assumption),
and we are done. 

So from now on we can assume that $\{i-1,i+1\}$ is crossed by
a non-diameter arc $\{i,u\}\in \mathcal{X}$. W.l.o.g. we can assume
that $u\in [i+2,i+n-1]$ (the other case $u\in [i+n+1,i-2]$ is
completely symmetric), and that $u$ is 'maximal' in the sense 
that there are no arcs $\{i,u'\}\in \mathcal{X}$ with
$u'\in [u+1,i+n-1]$. 

Now consider the arc $\{i-1,u\}$. It crosses 
the diameter $\{i,i+n\}_r\in \nc\nc\mathcal{X}$. 
Hence $\{i-1,u\}$ must be crossed by an element from $\mathcal{X}$. 

If $\{i-1,u\}$ is crossed by a diameter attached at $t\in [i,u-1]$
then we can conclude that $\{i,i+n\}_g\in \mathcal{X}$, as claimed.
In fact, no red such diameter $\{t,t+n\}_r\in \mathcal{X}$ can 
cross $\{i-1,u\}$ since otherwise condition (Pt3) implied that 
$\{i,i+n\}_r\in \mathcal{X}$, contradicting the assumption. But if
a green diameter $\{t,t+n\}_g\in \mathcal{X}$ 
crosses $\{i-1,u\}$ then again by (Pt3) we conclude that 
$\{i,i+n\}_g\in \mathcal{X}$. 

So we can assume that $\{i-1,u\}$ is crossed by a non-diameter 
arc $\{r,t\}\in \mathcal{X}$ where $r\in [i,u-1]$ (and then 
$t\in [u+1,i-2]$). 
 
By the choice of $u$ we know that $t\not\in [u+1,i+n-1]$
(because otherwise condition (Pt1) would give an arc 
$\{i,t\}\in\mathcal{X}$ longer than $\{i,u\}$).  

If $t=i+n$ then (Pt1) implies that $\{i,i+n\}_g\in \mathcal{X}$,
and we are done. 

So we can assume that $t\in [i+n+1,i-2]$. Condition (Pt1) then 
implies that $\{i,t\}\in\mathcal{X}$. Moreover, if now 
$t\in [i+n+1,u+n-1]$ then $\{i,t\}\in\mathcal{X}$ crosses
the rotated arc $\{i+n,u+n\}$ which is in $\mathcal{X}$
since $\mathcal{X}$ is invariant under 180 degree rotation. 
Then condition (Pt1) implies that the diameter 
$\{i,i+n\}_g\in \mathcal{X}$, and we are done in this case;
see Figure \ref{fig:picture2}. 

\begin{figure} 
$$  \begin{tikzpicture}[auto]
    \node[name=s, shape=regular polygon, regular polygon sides=22, minimum size=4cm, draw] {}; 
    \draw[shift=(s.corner 1)]  node[above]  {$i$};
    \draw[shift=(s.corner 7)] node[left] {$u$};
    \draw[shift=(s.corner 12)]  node[below] {$i+n$};
    \draw[shift=(s.corner 14)] node[below] {$t$};
    \draw[shift=(s.corner 18)]  node[right] {$u+n$};
    \draw[thick] (s.corner 1) to (s.corner 7);
    \draw[thick] (s.corner 12) to (s.corner 18);
    \draw[thick] (s.corner 1) to (s.corner 14);
    \draw[thick,dashed, green] (s.corner 1) to (s.corner 12);
  \end{tikzpicture} 
$$
\caption{~}\label{fig:picture2}
\end{figure}

Therefore we are left with the case that $\{i,t\}\in\mathcal{X}$
and $t\in [u+n,i-2]$. W.l.o.g. we can choose such $t$ whose
arc has maximal length, i.e. none of the arcs $\{i,t'\}$ with 
$t'\in [u+n,t-1]$ is in $\mathcal{X}$. 

Now we consider the arc (possibly a diameter) $\{u,t\}$. It crosses
the diameter $\{i,i+n\}_r$ which is in $\nc\nc\mathcal{X}$, so it
must be crossed by an element from $\mathcal{X}$. 

If $\{u,t\}$ is crossed by an arc from $\mathcal{X}$
attached at the vertex $i$ then
by the choice of $u$ and $t$ (maximality of the arcs $\{i,u\}$
and $\{i,t\}$) it can only be crossed by the green diameter 
$\{i,i+n\}_g\in\mathcal{X}$, and we are done. 

So we can assume that $\{u,t\}$ is crossed by an arc from $\mathcal{X}$
not attached at $i$. Any such arc necessarily also crosses 
$\{i,u\}$ or $\{i,t\}$. But then we are in one of the situations already
dealt with above where we have in each case concluded that 
the desired green arc $\{i,i+n\}_g$ is in $\mathcal{X}$ or 
obtained a contradiction. 

This completes the proof of the lemma. 
\end{proof}

After these two preparatory results we now come to the proof 
of the main result of this section. 
\medskip

\noindent
\underline{{\em Proof of Theorem \ref{thm:ncnc_vs_Ptolemy}.}}
The direction '(a) $\Longrightarrow$ (b)' is fairly straightforward. 
Let $\mathcal{X} = \nc\nc\mathcal{X}$. In each of the conditions
(Pt1), (Pt2), (Pt3) as visualized in Figure \ref{tab:Ptolemy-D-Condition}
we have to confirm that the dashed arcs must be 
in $\nc\nc\mathcal{X} = \mathcal{X}$. Slightly reformulated 
this means that for each of the dashed arcs 
the following holds: if they are crossed by an arc (possibly 
diameter) $\alpha$ of the $2n$-gon $P$ then $\alpha$ must also
cross an element in $\mathcal{X}$. But this condition is indeed
easily verified from looking at the figures provided in
Figure \ref{tab:Ptolemy-D-Condition} 
(where the solid arcs are in $\mathcal{X}$
by assumption). 
\smallskip

The direction '(b) $\Longrightarrow$ (a)' is much more involved.

Let $\mathcal{X}$ be a Ptolemy diagram of type $D$, 
so $\mathcal{X}$ satisfies conditions
(Pt1), (Pt2) and (Pt3). We have to show that 
$\mathcal{X} = \nc\nc\mathcal{X}$. Note that the inclusion
$\mathcal{X} \subseteq \nc\nc\mathcal{X}$ always holds (by definition
of the operator $\nc$). 

Thus we have to show that the conditions (Pt1), (Pt2) and (Pt3)
imply that $\nc\nc\mathcal{X}\subseteq\mathcal{X}$. 
This is where we shall make use of the preceding lemmas. 

We have to consider the cases of pairs of non-diameters and of
diameters separately. 

First, suppose (for a contradiction) there was a pair 
$\overline{\{i,j\}}$ of non-diameter arcs in 
$(\nc\nc\mathcal{X}) \setminus \mathcal{X}$. By Lemma
\ref{lem:non-diameter}\,(b), two diameters of the same colour
attached at $i$ and $j$ are then in $\mathcal{X}$, say
$\{i,i+n\}_g\in \mathcal{X}$ and $\{j,j+n\}_g\in \mathcal{X}$.

We can w.l.o.g. (by symmetry) assume that $j\in [i,i+n]$, i.e.
$i,j,i+n,j+n$ come in this order when going counterclockwise around 
around the boundary of $P$. 

We then consider the vertices $s\in [i+1,j-1]$, see Figure
\ref{fig:picture3}. We call such a vertex 
{\em free} if there is no arc $\{u,v\}\in\mathcal{X}$ with endpoints
$u\in [i,s-1]$ and $v\in [s+1,j]$. 
Such a free vertex must exist. In fact, if $i+1$ is
free, we are done. Otherwise, there exists an arc 
$\{i,s_1\}\in\mathcal{X}$ with $s_1\in [i+2,j-1]$ ($s_1\neq j$
since $\{i,j\}\not\in \mathcal{X}$ by assumption). Choose a longest
such arc; then the vertex $s_1$ is free since otherwise condition
(Pt1) would produce a longer arc than $\{i,s_1\}$ in $\mathcal{X}$
or implies that $\{i,j\}\in\mathcal{X}$, in both cases a contradiction.

Now take a free vertex $s\in [i+1,j-1]$ and consider the green 
diameter $\{s,s+n\}_g$. It crosses $\{i,j\}\in \nc\nc\mathcal{X}$, 
so it must be crossed by an element from $\mathcal{X}$. 

We claim that $\{s,s+n\}_g$ can not be crossed by a non-diameter
arc in $\mathcal{X}$. In fact, with $s$ being a free vertex the 
diameter $\{s,s+n\}_g$ can not be crossed by an arc from $\mathcal{X}$
having both endpoints 
in $[i,j]$. If $\{s,s+n\}_g$ was crossed by an arc having one
endpoint in $[i,j]$ and the other outside $[i,j]$ then such an arc
would cross one of the green diameters $\{i,i+n\}_g\in\mathcal{X}$ 
and $\{j,j+n\}_g\in\mathcal{X}$
and condition (Pt3) would produce an arc in $\mathcal{X}$ contradicting
the freeness of $s$ (use the rotational symmetry of $\mathcal{X}$ if
initially the freeness of $s+n$ is violated). 
By rotational symmetry all arguments apply equally well to the 
polygon bounded by $\{i+n,j+n\}$ and the edges of $P$
along the interval $[i+n,j+n]$, and the
free vertex $s+n$ therein. Thus the only possibility left is that
$\{s,s+n\}_g$ was crossed by a non-diameter arc $\{u,v\}\in\mathcal{X}$
with $u\in [j+1,i+n-1]$ and $v\in [j+n+1,i-1]$, see Figure 
\ref{fig:picture3}. 
Then $\{u,v\}$ crossed
both green diameters $\{i,i+n\}_g\in\mathcal{X}$ and 
$\{j,j+n\}_g\in\mathcal{X}$ and condition (Pt3) implies that
$\{i,u\}\in\mathcal{X}$. But the latter crosses 
$\{j,j+n\}_g\in\mathcal{X}$ and another application of (Pt3) 
shows that $\{i,j\}\in \mathcal{X}$, a contradiction. 

\begin{figure} 
$$  \begin{tikzpicture}[auto]
    \node[name=s, shape=regular polygon, regular polygon sides=22, minimum size=4cm, draw] {}; 
    \draw[shift=(s.corner 22)]  node[above]  {$i$};
    \draw[shift=(s.corner 2)] node[above] {$s$};
    \draw[shift=(s.corner 5)]  node[left] {$j$};
    \draw[shift=(s.corner 9)] node[left] {$u$};
    \draw[shift=(s.corner 11)]  node[below] {$i+n$};
    \draw[shift=(s.corner 13)]  node[below] {$s+n$};
    \draw[shift=(s.corner 16)] node[right] {$j+n$};
    \draw[shift=(s.corner 19)]  node[right] {$v$};
    \draw[thick] (s.corner 22) to (s.corner 9);
    \draw[thick] (s.corner 9) to (s.corner 19);
    \draw[thick,green] (s.corner 5) to (s.corner 16);
    \draw[thick,green] (s.corner 22) to (s.corner 11);
    \draw[thick,dotted,green] (s.corner 2) to (s.corner 13);
    \draw[thick,dotted] (s.corner 22) to (s.corner 5);
    \draw[thick,dotted] (s.corner 11) to (s.corner 16);
  \end{tikzpicture} 
$$
\caption{~}\label{fig:picture3}
\end{figure}

This completes the proof of the claim that the diameter 
$\{s,s+n\}_g$ can not be crossed by any non-diameter arc
in $\mathcal{X}$.  

Therefore, $\{s,s+n\}_g$ must be crossed by a diameter in $\mathcal{X}$
which is necessarily red, say by $\{s',s'+n\}_r$. If 
$s'\not\in \{i,j,i+n,j+n\}$ then condition (Pt2) implies the 
existence of a non-diameter arc in $\mathcal{X}$ crossing 
$\{s,s+n\}_g$. But this has just been excluded by the preceding claim. 
Finally, if $s'\in \{i,j,i+n,j+n\}$ then $\{i,i+n\}_r\in\mathcal{X}$
or $\{j,j+n\}_r\in\mathcal{X}$; but then condition (Pt2) implies
that $\{i,j\}\in\mathcal{X}$, contradicting the assumption. 

This completes the proof that there are no non-diameter arcs 
in $(\nc\nc\mathcal{X})\setminus \mathcal{X}$. 
\smallskip

Secondly, suppose (for a contradiction) that there was a 
(w.l.o.g.\;red) diameter $\{i,i+n\}_r$ in 
$(\nc\nc\mathcal{X}) \setminus \mathcal{X}$. 
By Lemma \ref{lem:diameter} we have that the green diameter
$\{i,i+n\}_g$ is in $\mathcal{X}$. 

We consider the green diameter $\{i+1,i+n+1\}_g$. It crosses 
the red diameter $\{i,i+n\}_r\in \nc\nc\mathcal{X}$, so 
$\{i+1,i+n+1\}_g$ must be crossed by an element from $\mathcal{X}$.

If $\{i+1,i+n+1\}_g$ is crossed by a (necessarily red) diameter
$\{s,s+n\}_r\in \mathcal{X}$, then $s\not\in \{i,i+n\}$ by 
assumption. But then $\{s,s+n\}_r$ crosses $\{i,i+n\}_g\in \mathcal{X}$ 
and condition (Pt2) implies that $\{i+1,i+n+1\}_g$ is also crossed by 
a non-diameter arc $\{i,s\}\in \mathcal{X}$. 

So from now on we can assume that $\{i+1,i+n+1\}_g$ is crossed 
by a non-diameter arc from $\mathcal{X}$. 
Using condition (Pt3) we can even assume that $\{i+1,i+n+1\}_g$ is 
crossed by a non-diameter arc of the form $\{i,u\}\in\mathcal{X}$ 
with $u\in [i+2,i+n-1]$. (In fact, any non-diameter arc not attached at 
$i$ crossing $\{i+1,i+n+1\}_g$ also crosses $\{i,i+n\}_g\in \mathcal{X}$ 
and (Pt3) can be applied.) W.l.o.g. we choose $u$ maximal with this property,
i.e. $\{i,r\}\not\in \mathcal{X}$ for all $r\in [u+1,i+n-1]$. 

Now consider the green diameter $\{u,u+n\}_g$, see Figure 
\ref{fig:picture4}.
\begin{figure} 
$$  \begin{tikzpicture}[auto]
    \node[name=s, shape=regular polygon, regular polygon sides=22, minimum size=4cm, draw] {}; 
    \draw[shift=(s.corner 1)]  node[above]  {$i$};
    \draw[shift=(s.corner 3)] node[above] {$i+1$};
    \draw[shift=(s.corner 7)]  node[left] {$u$};
    \draw[shift=(s.corner 11)] node[below] {$i+n$};
    \draw[shift=(s.corner 14)]  node[below] {$i+n+1$};
    \draw[shift=(s.corner 18)]  node[right] {$u+n$};
    \draw[thick] (s.corner 1) to (s.corner 7);
    \draw[thick] (s.corner 12) to (s.corner 18);
    \draw[thick,green] (s.corner 1) to (s.corner 12);
    \draw[thick,dashed, decorate,decoration=snake, red] (s.corner 1) to (s.corner 12);
    \draw[thick,dashed,green] (s.corner 7) to (s.corner 18);
    \draw[thick,dashed,green] (s.corner 2) to (s.corner 13);
  \end{tikzpicture} 
$$
\caption{~}\label{fig:picture4}
\end{figure}

It crosses the red diameter $\{i,i+n\}_r\in \nc\nc\mathcal{X}$, so 
$\{u,u+n\}_g$ must be crossed by an element from $\mathcal{X}$.

Suppose first that $\{u,u+n\}_g$ is crossed by a (red) diameter
$\{v,v+n\}_r\in \mathcal{X}$. Then one of the endpoints, say $v$,
must be in $[i+1,i+n-1]$ ($v=i$ is impossible by the assumption that
$\{i,i+n\}_r\not\in\mathcal{X}$). If $v\in [i+1,u-1]$ then
$\{v,v+n\}_r\in \mathcal{X}$ crosses $\{i,u\}\in \mathcal{X}$
and condition (Pt3) implies that $\{i,i+n\}_r\in\mathcal{X}$,
contradicting the assumption. If $v\in [u+1,i+n-1]$ then 
$\{v,v+n\}_r$ crosses $\{i,i+n\}_g\in \mathcal{X}$ and condition
(Pt2) yields an arc $\{i,v\}\in \mathcal{X}$ contradicting the
maximality of $u$. 

So we are left with the case that $\{u,u+n\}_g$ is crossed 
by a non-diameter arc $\{v,w\}\in \mathcal{X}$ (and hence also
by the rotated arc $\{v+n,w+n\}\in \mathcal{X}$). If none of the arcs
in the pair $\overline{\{v,w\}}$ crosses the green diameter 
$\{i,i+n\}_g$ then one of the arcs, say $\{v,w\}$ crosses 
$\{i,u\}\in \mathcal{X}$ and condition (Pt1) yields an arc 
contradicting the maximality of $u$ or the assumption that 
$\{i,i+n\}_r\not\in\mathcal{X}$. 

So we can assume from now on that the non-diameter arc
$\{v,w\}\in \mathcal{X}$ crossing $\{u,u+n\}_g$ also
crosses the green diameter $\{i,i+n\}_g\in\mathcal{X}$. 

Then one of the endpoints, say $v$, must be in the interval 
$[i+1,i+n-1]$, but $v\neq u$ (otherwise the arc can not cross
$\{u,u+n\}_g$).  

If $v\in [i+1, u-1]$ then $w\in [i+n+1,u+n-1]$. But then
$\{v,w\}\in\mathcal{X}$ crosses $\{i+n,u+n\}\in\mathcal{X}$
and condition (Pt1) implies that $\{v,i+n\}\in \mathcal{X}$;
but $\{v,i+n\}$ crosses $\{i,u\}\in \mathcal{X}$ and hence
condition (Pt1) yields that the red diameter $\{i,i+n\}_r\in 
\mathcal{X}$, contradicting the assumption. 

Finally, if $v\in [u+1, i+n-1]$ then $w\in [u+n+1,i-1]$. 
But then 
condition (Pt3) would imply
that $\{i,v\}\in \mathcal{X}$, contradicting 
the maximality of $u$. 
\smallskip

This completes the proof that there are no diameters 
in $(\nc\nc\mathcal{X})\setminus \mathcal{X}$. 
\smallskip

Together with the earlier proof hat there are no non-diameter
arcs in $(\nc\nc\mathcal{X})\setminus \mathcal{X}$ we have
shown that conditions (Pt1), (Pt2) and (Pt3) imply
that $\nc\nc\mathcal{X} = \mathcal{X}$. 
\smallskip

This completes the proof of Theorem \ref{thm:ncnc_vs_Ptolemy}.
\hfill $\Box$

\section{Counting torsion pairs in the cluster category of 
type~D}
\label{sec:counting}

In this section our aim is to count the torsion pairs in the
cluster category of Dynkin type $D_n$. As a main result we will 
give the generating function for the number of torsion pairs 
explicitly. 
This will be achieved by first providing an alternative 
description of Ptolemy diagrams of type $D$. In this description 
we will build on Ptolemy diagrams of Dynkin type $A$, as introduced
in \cite{HJR-Ptolemy}. We briefly recall the definition and the
facts needed for our purposes. 

For any $n\ge 1$, let $P$ be a regular $(n+3)$-gon with a 
distinguished oriented edge which we
refer to as the {\em distinguished base edge}.   
An {\em edge} of $P$ is a set of two neighbouring vertices of the polygon.  
As before, an {\em arc} is a set
of two non-neighbouring vertices of $P$.  

Two arcs $\{i,j\}$ and $\{k,\ell\}$
\Dfn{cross} if their end points are all distinct and come in the
order $i, k, j, \ell$ when moving around the
polygon $P$ in one direction or the other.  This corresponds to an
obvious notion of geometrical crossing.  Note that an arc does
not cross itself and that two arcs sharing an end point do not
cross.

  Let $\mathcal{X}$ be a set of arcs in $P$.  Then $\mathcal{X}$ is a
  \Dfn{Ptolemy diagram of type~$A$} if it has the following property:
  when $\{i,j \}$ and $\{k,\ell\}$ are crossing arcs from $\mathcal{X}$, 
  then those of $\{i,k\}$, $\{i,\ell \}$, $\{j,k\}$, $\{j,\ell \}$ 
  which are arcs are also in $\mathcal{X}$,
  see Figure \ref{fig:Ptolemy-A-Condition}.

\begin{figure}
\[
  \begin{tikzpicture}[auto]
    \node[name=s, shape=regular polygon, regular polygon sides=20, minimum size=4cm, draw] {}; 
    \draw[thick] (s.corner 5) to node[very near start,below=20pt] {$~$} (s.corner 16);
    \draw[shift=(s.corner 5)] node[left] {$i$};
    \draw[shift=(s.corner 16)] node[right] {$j$};
    \draw[thick] (s.corner 9) to node[near start] {$~$} (s.corner 19);
    \draw[shift=(s.corner 9)] node[left] {$k$};
    \draw[shift=(s.corner 19)] node[right] {$\ell$};
    \draw[thick,dotted] (s.corner 5) to (s.corner 9);
    \draw[thick,dotted] (s.corner 5) to (s.corner 19);
    \draw[thick,dotted] (s.corner 16) to (s.corner 9);
    \draw[thick,dotted] (s.corner 16) to (s.corner 19);
  \end{tikzpicture} 
\]
\caption{The Ptolemy condition in Dynkin type~$A$}
\label{fig:Ptolemy-A-Condition}
\end{figure}
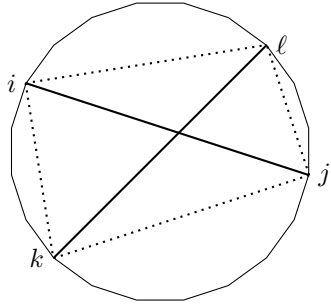

In \cite{HJR-Ptolemy} we have classified and enumerated the
Ptolemy diagrams of type $A_n$. In particular, we have shown
in \cite[Section 3]{HJR-Ptolemy} that
the generating function for Ptolemy diagrams of type
$A$,
\begin{equation}
  \label{eq:Ptolemy-A-generating-function}
  \P_A(y):=\sum_{N\geq1} \#\{\text{Ptolemy diagrams of type~$A$ of the $(N+1)$-gon}\}
  y^N.  
\end{equation}
satisfies
\begin{equation*}
  \P_A(y) = y + \frac{\P_A(y)^2}{1-\P_A(y)} +
  \frac{\P_A(y)^3}{1-\P_A(y)}.
\end{equation*}

\smallskip

We now turn back to Ptolemy diagrams of type $D$. 
In this section we shall determine the generating function for 
Ptolemy diagrams of type $D$,
\begin{equation*}
  \label{eq:Ptolemy-D-generating-function}
  \P_D(y):=\sum_{N\geq1} \#\{\text{Ptolemy diagrams of type~$D$ of the $2N$-gon}\}
  y^N.
\end{equation*}
Roughly speaking we will decompose a Ptolemy diagram of type~$D$ into
a \lq central region\rq\ (again with 180\textdegree\ rotational
symmetry) containing all the diameters, bounded by a polygon and a
circular arrangement of Ptolemy diagrams of type~$A$ \lq glued\rq\ to
the edges of this central polygon as sketched in
Figure~\ref{fig:decomposition}.
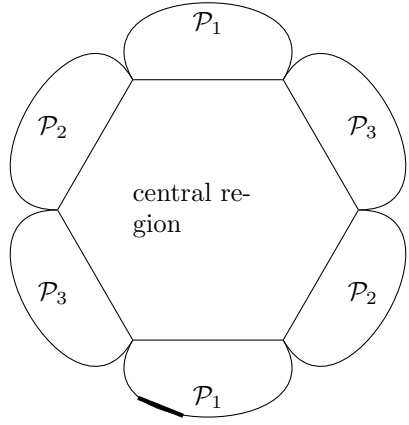
\begin{figure}
 \centering
  \begin{tikzpicture}[auto]
    \node[name=s, shape=regular polygon, regular polygon sides=6,  minimum size=4cm, draw, text width=2cm] at (0,0) {central region};
    \draw[out=60, in=120,looseness=2] (s.corner 1) to node {$\P_1$}(s.corner 2);
    \draw[out=120,in=180,looseness=2] (s.corner 2) to node {$\P_2$}(s.corner 3);
    \draw[out=180,in=240,looseness=2] (s.corner 3) to node {$\P_3$}(s.corner 4);
    \draw[out=240,in=300,looseness=2] (s.corner 4) to node {$\P_1$}(s.corner 5);
    \draw[out=300,in=0,  looseness=2] (s.corner 5) to node {$\P_2$}(s.corner 6);
    \draw[out=0,  in=60, looseness=2] (s.corner 6) to node {$\P_3$}(s.corner 1);
    \draw[ultra thick] (-0.93, -2.5) to (-0.33,-2.74);
  \end{tikzpicture} 
  \caption{The decomposition of a Ptolemy diagram of type $D$ with
    the distinguished base edge drawn bold.}
  \label{fig:decomposition}
\end{figure}

Before we can define the \lq central region\rq\ of a Ptolemy diagram
of type~$D$ precisely, we need the following fact about the structure
of Ptolemy diagrams:
\begin{Lemma}\label{lem:uncrossed-diameter}
  Let $\mathcal{X}$ be a Ptolemy diagram of type~$D$ in the $2n$-gon
  $P$.  Suppose that $\mathcal{X}$ does not contain a diameter.  Then
  there exists a diameter $\{i, i+n\}$ (green or red) which is not
  crossed by any arc in $\mathcal{X}$.
\end{Lemma}
\begin{proof}
  Let $\{i,j\}\in\mathcal{X}$ (with $j\in[i,i+n]$) such that the set
  of vertices $[i,j]$ is maximal.  We show that $\{i,i+n\}$ is not
  crossed by any arc in $\mathcal{X}$.

  Suppose that an arc $\{k,\ell\}\in\mathcal{X}$ (with
  $\ell\in[k,k+n]$) crosses $\{i,i+n\}$.  There are two cases to
  distinguish: if $k$, $i$, $j$, $\ell$ come in this order when going
  counterclockwise then $[k, \ell]\supsetneq[i,j]$, contradicting the
  maximality assumption.  On the other hand, if the order is $k$,
  $i$, $\ell$, $j$, then $\{k,\ell\}$ and $\{i,j\}$ cross and
  condition (Pt1) implies that $\{k,j\}\in\mathcal{X}$.  Now, if
  $j\in[k,k+n]$ then $[k,j]\supsetneq[i,j]$ contradicts the
  maximality assumption.  On the other hand, if $j+n\in[k,k+n]$ we
  have that $\{k, j\}$ and $\{j+n,i+n\}$ cross.  Condition (Pt1) then
  forces $\{j,j+n\}\in\mathcal{X}$, contradicting the hypothesis of
  the lemma.
\end{proof}

\begin{Definition}
  Let $\mathcal{X}$ be a Ptolemy diagram of type~$D$ in the $2n$-gon
  $P$.  Then the \Dfn{central region} is a set of vertices and arcs
  of $\mathcal{X}$ constructed as follows:

  Suppose that $\mathcal{X}$ does not contain any diameter.  Then, by
  Lemma~\ref{lem:uncrossed-diameter} there is a diameter $\{i,i+n\}$
  (green or red) which is not crossed by any arc in $\mathcal{X}$.
  Let $\fV$ be the shortest sequence of vertices $( i = i_0, i_1,
  \dots, i_k = i+n )$ such that $i_j$ and $i_{j+1}$ are connected by
  an edge of $P$ or an arc in $\mathcal{X}$ for $0\leq j\leq k-1$.

  If $\mathcal{X}$ does contain a diameter 
  $ \{i,i+n \}$ (green or red) then
  let $\fV$ be the shortest sequence of vertices $(i =
  i_0, i_1, \dots, i_k = i+n )$ such that
  $i_j$ and $i_{j+1}$ are connected by an edge of $P$ or an arc  
  in $\mathcal{X}$ for $0\leq j\leq k-1$ and that contains 
  one end point of every diameter from $\mathcal{X}$.

  The \Dfn{central region} of $\mathcal{X}$ is the polygon containing
  the vertices in $\fV$, their opposite vertices and the edges of $P$ and
  arcs in $\mathcal{X}$ connecting vertices in $\fV$.  We say that
  the edges and arcs $\{ i_j, i_{j+1} \}$ and $\{
  i_j+n, i_{j+1}+n \}$ for $0\leq j\leq k-1$
  \Dfn{bound} the central region.
\end{Definition}

\begin{Lemma}\label{lem:decomposition}
  In a Ptolemy diagram of type~$D$ there is no arc crossing one
  of the edges or arcs bounding the central region.
  Consequently, the diagrams attached to the central region are
  Ptolemy diagrams of type~$A$.
\end{Lemma}
\begin{proof}
  Let $\fV$ be the set of vertices on the boundary of the central
  region.  Let $i$ and $j$ be two vertices in $\fV$ and let
  $\{i,j\}\in\mathcal{X}$ (with $j\in[i,i+n]$).  Suppose that the arc
  $\{i,j\}$ is crossed by an arc $\{k,\ell\}\in\mathcal{X}$ (with
  $\ell\in[k,k+n]$).  Let us assume without loss of generality that
  $k\in[i,j]$.  We show that this forces that a diameter
  $\{k,k+n\}\in\mathcal{X}$, so the arc $\{i,j\}$ cannot be on the
  boundary of the central region.

  If $\{k,\ell\}$ is a diameter there is nothing to show, so we
  suppose that this is not the case.  If $\{k,\ell\}$ is crossed by a
  diameter then condition (Pt3) implies that
  $\{k,k+n\}\in\mathcal{X}$ and we are done.  

  Suppose now that $\{k,\ell\}$ is not crossed by a diameter.  In
  particular, $\{i,j\}$ is not a diameter either.  Condition (Pt1)
  then implies that $\{i,\ell\}\in\mathcal{X}$.  It follows that
  $\ell\not\in\fV$ because otherwise the sequence $\fV$ is not minimal:
  replacing the vertices in $[i,\ell]$ by just $i$ and $\ell$ yields
  a shorter sequence (at least the vertex $j$ does not appear).

  However, if $\ell\not\in\fV$, there must be an arc
  $\{r,s\}\in\mathcal{X}$ with $r,s\in\fV$ that crosses $\{k,\ell\}$
  and such that $i,j,r,s$ appear in this order when going around the
  polygon counterclockwise.  Since this arc cannot be a diameter by
  assumption, condition (Pt1) implies that $\{i,s\}\in\mathcal{X}$,
  which again violates the minimality of $\fV$.

  The fact that the diagrams attached to the central region are
  Ptolemy diagrams of type~$A$ follows, because such a component
  cannot contain any diameters and condition (Pt1) coincides with the
  Ptolemy condition in type~$A$.
\end{proof}

\begin{Proposition}\label{prop:decomposition}
  Let 
  $$
  \C(y)=\sum_{k\geq0} \#\{\text{central regions with $2k+2$ bounding edges}\} y^k
  $$
  be the generating function for central regions.  Then the
  generating function for Ptolemy diagrams of type~$D$ equals
  $$
  \P_D(y) = y\P_A^\prime(y)\C\big(\P_A(y)\big).
  $$
\end{Proposition}
\begin{proof}
  Suppose we are given a central region with $2k+2$ bounding edges.
  We can then construct every Ptolemy diagram of type~$D$ with this
  central region in a unique way from a list $(\mathcal{X}_0, 
  \mathcal{X}_1,\dots,
  \mathcal{X}_k)$ of Ptolemy diagrams of type~$A$ together with an additional
  distinguished edge $\fd$ in $\mathcal{X}_0$.  As we will see we have to
  insist that $\fd$ is different from the distinguished base edge of
  $\mathcal{X}_0$ except if $\mathcal{X}_0$ is the degenerate Ptolemy 
  diagram of type~$A$.

  Namely, to construct a diagram from the given data we glue in
  clockwise order $\mathcal{X}_0, \mathcal{X}_1,\dots, \mathcal{X}_k, 
  \mathcal{X}_0, \mathcal{X}_1, \dots,
  \mathcal{X}_k$ onto the bounding edges of the central region along their
  respective distinguished base edges.  Finally, we declare the edge
  in the resulting diagram that corresponds to $\fd$ in one copy of
  $\mathcal{X}_0$ in $\fD$ to be the distinguished base edge of 
  $\mathcal{X}$. Because of the symmetry of the resulting diagram it 
  does not matter which of the two copies of $\mathcal{X}_0$ we select.

  We now notice that the generating function for Ptolemy diagrams of
  type~$A$ with an additional distinguished edge satisfying the
  condition mentioned above is $y \P_A^\prime(y)$, where the prime
  denotes the derivative.  Thus, the generating function for Ptolemy
  diagrams of type~$D$ with a central region of size $2k+2$ equals
  \[
  \#\{\text{central regions with $2k+2$ bounding edges}\}\; y
  \P_A^\prime(y) \P_A(y)^k.
  \]
  Summing over all $k$ we obtain the claim.

  We remark that we defined the degenerate Ptolemy diagram of type
  $D$ with two vertices such as to make this decomposition agree with
  Condition (Pt1).
\end{proof}

We now distinguish three kinds of central regions occuring in Ptolemy
diagrams of type~$D$.  Although not imperative for the rest of this
article these are chosen in a way such that the set of diagrams of
each kind is closed under the operator $\nc$ and under rotation.  Note that
the central region of a Ptolemy diagram of type $D$ is itself a
Ptolemy diagram of type $D$.  Thus we actually define three kinds of
Ptolemy diagrams, according to the kind of their central regions:

\begin{Definition}
  Let $\mathcal{X}$ be a Ptolemy diagram of type~$D$.  We say that two
  diameters in $\mathcal{X}$ are \Dfn{paired} if they connect the same two
  vertices (and are thus of different colour).

  Every Ptolemy diagram $\mathcal{X}$ of type $D$ 
  is of one of the following three kinds:
  \begin{enumerate}
    \renewcommand{\labelenumi}{(\Roman{enumi})}
  \item All diameters (if any) in $\mathcal{X}$ are paired.

  \item All diameters in $\mathcal{X}$ are of the same colour and there 
  are at least two diameters in $\mathcal{X}$.

  \item Not all diameters in $\mathcal{X}$ are paired and if there are
    several diameters, both colours occur.
  \end{enumerate}
\end{Definition}

We remark that according to the construction in the proof of
Proposition~\ref{prop:decomposition} the type of a Ptolemy diagram is
determined by its central region.  In
Figures~\ref{fig:typeI}--\ref{fig:typeIII} the central regions with
few bounding edges are listed.

In the remainder of this section we describe each of the three kinds
of central regions precisely and determine in each case 
the corresponding generating function.

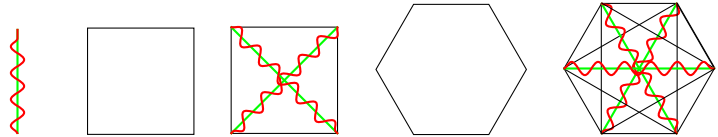
\begin{figure}
  \begin{tikzpicture}[auto]
    \draw[thick,green] (0,0) to (0,1.4);
    \draw[thick,decorate,decoration=snake,red] (0,0) to (0,1.4);
  \end{tikzpicture}
  \qquad
  \begin{tikzpicture}[auto]
    \node[name=s, shape=regular polygon, regular polygon sides=4,
    minimum size=2cm, draw] at (0,0) {};
  \end{tikzpicture}
  \quad
  \begin{tikzpicture}[auto]
    \node[name=s, shape=regular polygon, regular polygon sides=4,
    minimum size=2cm, draw] at (0,0) {};
    \draw[thick,green] (s.corner 1) to (s.corner 3);
    \draw[thick,decorate,decoration=snake,red] (s.corner 1) to (s.corner 3);
    \draw[thick,green] (s.corner 2) to (s.corner 4);
    \draw[thick,decorate,decoration=snake,red] (s.corner 2) to (s.corner 4);
  \end{tikzpicture}
  \quad
  \begin{tikzpicture}[auto]
    \node[name=s, shape=regular polygon, regular polygon sides=6,
    minimum size=2cm, draw] at (0,0) {};
  \end{tikzpicture}
  \quad
  \begin{tikzpicture}[auto]
    \node[name=s, shape=regular polygon, regular polygon sides=6,
    minimum size=2cm, draw] at (2,0) {};
    \draw (s.corner 1) to (s.corner 3);
    \draw[thick,green] (s.corner 1) to (s.corner 4);
    \draw[thick,decorate,decoration=snake,red] (s.corner 1) to (s.corner 4);
    \draw (s.corner 1) to (s.corner 5);
    \draw (s.corner 1) to (s.corner 6);
    \draw (s.corner 2) to (s.corner 4);
    \draw[thick,green] (s.corner 2) to (s.corner 5);
    \draw[thick,decorate,decoration=snake,red] (s.corner 2) to (s.corner 5);
    \draw (s.corner 2) to (s.corner 6);
    \draw (s.corner 3) to (s.corner 5);
    \draw[thick,green] (s.corner 3) to (s.corner 6);
    \draw[thick,decorate,decoration=snake,red] (s.corner 3) to (s.corner 6);
    \draw (s.corner 4) to (s.corner 6);
  \end{tikzpicture}
  \caption{The central regions of type (I) with at most 6 bounding
    edges.}
  \label{fig:typeI}
\end{figure}

\begin{figure}
  \begin{tikzpicture}[auto]
    \node[name=s, shape=regular polygon, regular polygon sides=4,
    minimum size=2cm, draw] at (0,0) {};
    \draw[thick,green] (s.corner 1) to (s.corner 3);
    \draw[thick,green] (s.corner 2) to (s.corner 4);
  \end{tikzpicture}
  \quad
  \begin{tikzpicture}[auto]
    \node[name=s, shape=regular polygon, regular polygon sides=6,
    minimum size=2cm, draw] at (2,0) {};
    \draw[thick,green] (s.corner 2) to (s.corner 5);
    \draw[thick,green] (s.corner 3) to (s.corner 6);
  \end{tikzpicture}
  \quad
  \begin{tikzpicture}[auto]
    \node[name=s, shape=regular polygon, regular polygon sides=6,
    minimum size=2cm, draw] at (2,0) {};
    \draw[thick,green] (s.corner 1) to (s.corner 4);
    \draw[thick,green] (s.corner 3) to (s.corner 6);
  \end{tikzpicture}
  \quad
  \begin{tikzpicture}[auto]
    \node[name=s, shape=regular polygon, regular polygon sides=6,
    minimum size=2cm, draw] at (2,0) {};
    \draw[thick,green] (s.corner 1) to (s.corner 4);
    \draw[thick,green] (s.corner 2) to (s.corner 5);
  \end{tikzpicture}

  \bigskip
  \begin{tikzpicture}[auto]
    \node[name=s, shape=regular polygon, regular polygon sides=6,
    minimum size=2cm, draw] at (2,0) {};
    \draw[thick,green] (s.corner 1) to (s.corner 4);
    \draw[thick,green] (s.corner 2) to (s.corner 5);
    \draw[thick,green] (s.corner 3) to (s.corner 6);
  \end{tikzpicture}
  \quad
  \begin{tikzpicture}[auto]
    \node[name=s, shape=regular polygon, regular polygon sides=6,
    minimum size=2cm, draw] at (2,0) {};
    \draw[thick,green] (s.corner 1) to (s.corner 4);
    \draw[thick,green] (s.corner 2) to (s.corner 5);
    \draw[thick,green] (s.corner 3) to (s.corner 6);
    \draw (s.corner 2) to (s.corner 4);
    \draw (s.corner 1) to (s.corner 5);
  \end{tikzpicture}
  \quad
  \begin{tikzpicture}[auto]
    \node[name=s, shape=regular polygon, regular polygon sides=6,
    minimum size=2cm, draw] at (2,0) {};
    \draw[thick,green] (s.corner 1) to (s.corner 4);
    \draw[thick,green] (s.corner 2) to (s.corner 5);
    \draw[thick,green] (s.corner 3) to (s.corner 6);
    \draw (s.corner 1) to (s.corner 3);
    \draw (s.corner 4) to (s.corner 6);
  \end{tikzpicture}
  \quad
  \begin{tikzpicture}[auto]
    \node[name=s, shape=regular polygon, regular polygon sides=6,
    minimum size=2cm, draw] at (2,0) {};
    \draw[thick,green] (s.corner 1) to (s.corner 4);
    \draw[thick,green] (s.corner 2) to (s.corner 5);
    \draw[thick,green] (s.corner 3) to (s.corner 6);
    \draw (s.corner 2) to (s.corner 6);
    \draw (s.corner 3) to (s.corner 5);
  \end{tikzpicture}
  \caption{The central regions of type (II) with green diameters with
    at most 6 bounding edges.}
  \label{fig:typeII}
\end{figure}
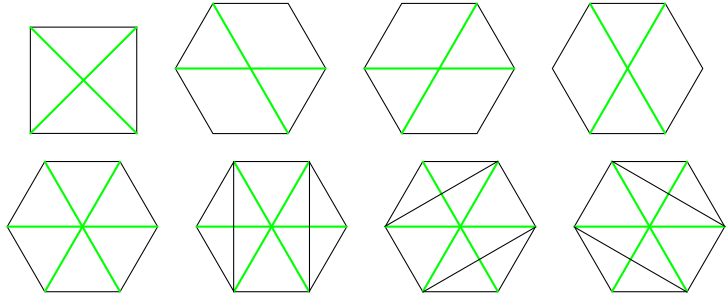

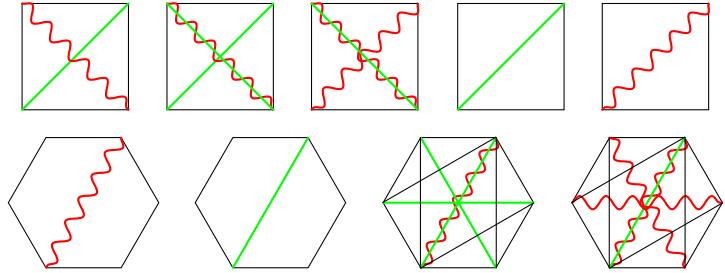
\begin{figure}
  \begin{tikzpicture}[auto]
    \node[name=s, shape=regular polygon, regular polygon sides=4,
    minimum size=2cm, draw] at (0,0) {};
    \draw[thick,green] (s.corner 1) to (s.corner 3);
    \draw[thick,decorate,decoration=snake,red] (s.corner 2) to (s.corner 4);
  \end{tikzpicture}
  \quad
  \begin{tikzpicture}[auto]
    \node[name=s, shape=regular polygon, regular polygon sides=4,
    minimum size=2cm, draw] at (0,0) {};
    \draw[thick,green] (s.corner 1) to (s.corner 3);
    \draw[thick,decorate,decoration=snake,red] (s.corner 2) to (s.corner 4);
    \draw[thick,green] (s.corner 2) to (s.corner 4);
  \end{tikzpicture}
  \quad
  \begin{tikzpicture}[auto]
    \node[name=s, shape=regular polygon, regular polygon sides=4,
    minimum size=2cm, draw] at (0,0) {};
    \draw[thick,decorate,decoration=snake,red] (s.corner 1) to (s.corner 3);
    \draw[thick,decorate,decoration=snake,red] (s.corner 2) to (s.corner 4);
    \draw[thick,green] (s.corner 2) to (s.corner 4);
  \end{tikzpicture}
  \quad
  \begin{tikzpicture}[auto]
    \node[name=s, shape=regular polygon, regular polygon sides=4,
    minimum size=2cm, draw] at (0,0) {};
    \draw[thick,green] (s.corner 1) to (s.corner 3);
  \end{tikzpicture}
  \quad
  \begin{tikzpicture}[auto]
    \node[name=s, shape=regular polygon, regular polygon sides=4,
    minimum size=2cm, draw] at (0,0) {};
    \draw[thick,decorate,decoration=snake,red] (s.corner 1) to (s.corner 3);
  \end{tikzpicture}

  \bigskip
  \begin{tikzpicture}[auto]
    \node[name=s, shape=regular polygon, regular polygon sides=6,
    minimum size=2cm, draw] at (0,0) {};
    \draw[thick,decorate,decoration=snake,red] (s.corner 1) to (s.corner 4);
  \end{tikzpicture}
  \quad
  \begin{tikzpicture}[auto]
    \node[name=s, shape=regular polygon, regular polygon sides=6,
    minimum size=2cm, draw] at (0,0) {};
    \draw[thick,green] (s.corner 1) to (s.corner 4);
  \end{tikzpicture}
  \quad
  \begin{tikzpicture}[auto]
    \node[name=s, shape=regular polygon, regular polygon sides=6,
    minimum size=2cm, draw] at (0,0) {};
    \draw[thick,decorate,decoration=snake,red] (s.corner 1) to (s.corner 4);
    \draw[thick,green] (s.corner 1) to (s.corner 4);
    \draw[thick,green] (s.corner 2) to (s.corner 5);
    \draw[thick,green] (s.corner 3) to (s.corner 6);
    \draw (s.corner 1) to (s.corner 3);
    \draw (s.corner 2) to (s.corner 4);
    \draw (s.corner 4) to (s.corner 6);
    \draw (s.corner 1) to (s.corner 5);
  \end{tikzpicture}
  \quad
  \begin{tikzpicture}[auto]
    \node[name=s, shape=regular polygon, regular polygon sides=6,
    minimum size=2cm, draw] at (0,0) {};
    \draw[thick,decorate,decoration=snake,red] (s.corner 1) to (s.corner 4);
    \draw[thick,green] (s.corner 1) to (s.corner 4);
    \draw[thick,decorate,decoration=snake,red] (s.corner 2) to (s.corner 5);
    \draw[thick,decorate,decoration=snake,red] (s.corner 3) to (s.corner 6);
    \draw (s.corner 1) to (s.corner 3);
    \draw (s.corner 2) to (s.corner 4);
    \draw (s.corner 4) to (s.corner 6);
    \draw (s.corner 1) to (s.corner 5);
  \end{tikzpicture}
  \caption{The central regions of type (III) with at most 6 bounding
    edges, up to rotation.}
  \label{fig:typeIII}
\end{figure}

\subsection{Ptolemy diagrams of the first kind.}
Let $\mathcal{X}$ be a central region with $2k+2$ bounding edges.  
If all
diameters in $\mathcal{X}$ are paired then by Condition (Pt2)
all arcs connecting end points of the diameters in $\mathcal{X}$ 
are also in $\mathcal{X}$.  
In particular, since $\mathcal{X}$ is a central region, if $\mathcal{X}$
contains a diameter then all $2k+2$ possible diameters ($k+1$ of each
colour) must be present. 

Finally, again because $\mathcal{X}$ is a central region, if 
$\mathcal{X}$ does not
contain a diameter then $\mathcal{X}$ is a polygon without any
(internal)  
arcs.  In this case $k$ must be greater than zero since the 
central region with only two vertices contains both diameters.

Thus, for $k=0$ there is one central region of the first kind while
for $k\geq1$ there are two.  In other words, the generating function
for central regions of the first kind equals
\begin{equation}
  \label{eq:1}
  \C_{D,I}(y) = \frac{1+y}{1-y}\tag{I}
  = 1 + 2y + 2y^2+2y^3+2y^4+2y^5+\dots\notag,
\end{equation}

Clearly, for $k\geq 1$ the operator $\nc$ maps the central region of
the first kind without diameters to the central region of the first
kind with all diameters and vice versa.  The degenerate central
region with two vertices is a fixed point under the operation of
$\nc$.

\subsection{Ptolemy diagrams of the second kind.}
We now consider the case that there are at least two diameters in
$\mathcal{X}$ and all of them are of the same colour.  Suppose that there is
a non-diameter arc $\{i, j \}$ crossing a
diameter.  Let us traverse the polygon starting at $i$ and
ending at $j$, where the direction of travel is chosen in such
a way that $j$ is encountered before $j+n$.  Let
$\fB$ be the set of end points of diameters in $\mathcal{X}$, including
$i$ and $j$ encountered in this way.  Then
Condition (Pt1) and Condition (Pt3) imply
that precisely the arcs connecting vertices in $\fB$ are present
in $\mathcal{X}$.
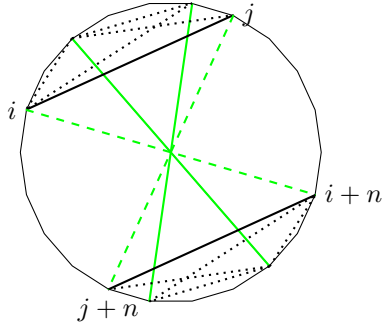
\[
  \begin{tikzpicture}[auto]
    \node[name=s, shape=regular polygon, regular polygon sides=22, minimum size=4cm, draw] {}; 
    \draw[shift=(s.corner 6)]  node[left] {$i$};
    \draw[shift=(s.corner 22)] node[right] {$j$};
    \draw[shift=(s.corner 17)]  node[right] {$i+n$};
    \draw[shift=(s.corner 11)]  node[below] {$j+n$};
    \draw (s.corner 4) to (s.corner 15);
    \draw[thick, green] (s.corner 4) to (s.corner 15);
    \draw[thick, green] (s.corner 1) to (s.corner 12);
    \draw[thick] (s.corner 6) to node[near start, below right=20pt] {$~$} (s.corner 22);
    \draw[thick] (s.corner 17) to (s.corner 11);
    \draw[thick,dotted] (s.corner 4) to (s.corner 6);
    \draw[thick,dotted] (s.corner 4) to (s.corner 22);
    \draw[thick,dotted] (s.corner 4) to (s.corner 1);
    \draw[thick,dotted] (s.corner 1) to (s.corner 6);
    \draw[thick,dotted] (s.corner 15) to (s.corner 11);
    \draw[thick,dotted] (s.corner 15) to (s.corner 17);
    \draw[thick,dotted] (s.corner 15) to (s.corner 12);
    \draw[thick,dotted] (s.corner 12) to (s.corner 17);
    \draw[thick,dashed, green] (s.corner 6) to (s.corner 17);
    \draw[thick,dashed, green] (s.corner 22) to (s.corner 11);
  \end{tikzpicture} 
\]

Furthermore, by Condition (Pt3) and
Condition (Pt1) there are at least two diameters in
$\mathcal{X}$ which are not crossed by another diagonal, and only 
arcs connecting end points of diameters in $\mathcal{X}$ can cross a 
diameter in $\mathcal{X}$.  

It remains to derive the generating function of central regions of
the second kind.  Such a central region is determined by a selection
of at least two diameters that are not crossed by any arcs and a
subset of these diameters that are crossed by arcs.  However,
end points of diameters that are chosen to be crossed by arcs
can only be neighboured by end points of other diameters in the
selected set.  Therefore, the central regions can be encoded by words
from an alphabet with three letters $o$ (not selected), $l$
(diameter) and $x$ (crossed diameter) such that $o$ and $x$ are not
consecutive, where we consider the first and the last letter of the
word adjacent and $l$ occurs at least twice.

One way to perform the computation of the generating function
\begin{equation*}
  \W(y)=\sum_{k\geq2} \#\{\text{words of length $k$}\} y^k
\end{equation*}
of such words is as follows: let $\W_o$ (respectively $\W_x$) be the
set of words that do not contain the sequence $ox$ or $xo$ and end in
$o$ (respectively $x$).  Furthermore, let $\W^\prime$ be the set of
all words that do not contain the sequence $ox$ or $xo$.  Then we
have the following combinatorial equations, omitting parenthesis
around singleton sets to improve readability:
\begin{align*}
  \W_o &= o + \W^\prime \cdot l\cdot o + \W_o\cdot o\\
  \W_x &= x + \W^\prime \cdot l\cdot x + \W_x\cdot x\\
  \W^\prime &= \emptyset + (\emptyset + \W^\prime\cdot l)\cdot(l + o
  + x) + \W_o\cdot(l+o) + \W_x\cdot(l+x)\\
  \W &= l \cdot \W^\prime \cdot l \cdot (\emptyset + o^+ +  x^+) \\
  &\phantom{=} + o^+\cdot l \cdot \W^\prime \cdot l \cdot o^* %
  + x^+\cdot l \cdot \W^\prime \cdot l \cdot x^*
\end{align*}
In these equations $+$ denotes the union of sets, $\F\cdot\G$ is the
set of all words obtained by appending a word from $\G$ to a word
from $\F$, $\emptyset$ is the empty word, $a^*$ denotes the set of
words composed of the letter $a$ only, including the empty word, and
$a^+$ equals $a^*$ without the empty word.

Passing to generating functions (we assign every letter the weight
$y$) and solving the system of equations we obtain $\W^\prime(y) =
\frac{1+y}{1-2y-y^2}$ and $\W(y) =
\frac{y^2(1+y)(1+2y-y^2)}{(1-y)^2(1-2y-y^2)}$.  Since we have to
choose one of two colours for the diameters we conclude that the
generating function for central regions of the second kind equals
\begin{align*}
  \label{eq:2}
  \C_{D,II} &= 2\frac{\W(y)}{y} = 2\frac{y(1+y)(1+2y-y^2)}{(1-y)^2(1-2y-y^2)}
  \tag{II}\\
  &=2y  + 14y^2  + 50y^3  + 142y^4 + 370y^5 + \dots\notag
\end{align*}

The action of the operator $\nc$ is most easily explained by its
effect on the corresponding words: if $\mathcal{X}$ is a central region
corresponding to a word $w$ then the word corresponding $\nc\mathcal{X}$ 
is obtained from $w$ by interchanging the letters $o$ and $x$.

\subsection{Ptolemy diagrams of the third kind.}
Suppose first that not all diameters in $\mathcal{X}$ are paired and both
colours occur.  In this case, if there is a paired diameter in 
$\mathcal{X}$ then all other diameters must be of the same colour.

Namely, let $\fa=\{i, i+n \}$ be a paired diameter in
$\mathcal{X}$, and let 
$\fb = \{ j, j+n \}$ and $\fc$ be unpaired
diameters of different colours, say $\fb$ red and $\fc$ green.  Then
Condition (Pt2) applied to $\fa$ and $\fb$ implies that
one of $\{ i, j \}$ and $\{ i, j+n \}$
crosses $\fc$.  But then Condition (Pt3) applied to this
crossing forces the presence of the green diameter pairing $\fb$, a
contradiction.

Furthermore, if there is a paired diameter in $\mathcal{X}$ 
and all other diameters are of the same colour then on both 
sides of the paired diameter the end points of all diameters 
are all connected by
Condition (Pt2) and Condition (Pt1).
However, the paired diameter itself is not crossed by any
non-diameter arcs, since Condition (Pt3) would then
force all diameters to be paired.

Otherwise, if there is no paired diameter then $\mathcal{X}$ 
contains at most
two diameters.  Suppose on the contrary that $\fa=\{i,i+n \}_r$ is 
a red unpaired diameter and $\fb = \{j,j+n \}_g$ and $\fc$ are 
both unpaired green diameters.  Then
Condition (Pt2) applied to $\fa$ and $\fb$ forces the
presence of an arc $\fd$ crossing $\fc$.  In turn,
Condition (Pt3) applied to $\fc$ and $\fd$ implies that
$\fa$ must be paired, contradicting our assumption.

As in Ptolemy diagrams of the first and second kind no other
arcs can cross a diameter.

Thus, central regions of the third kind having one paired diameter or
a single (unpaired) diameter can be constructed by choosing a colour
for the unpaired diameter(s) and one diameter in a polygon with at
least four edges.  Finally, there are two additional central region
with four vertices and two unpaired diameters of different colours:
\begin{align*}
  \label{eq:3}
  \C_{D,III} &= 2y + 4 \sum_{k\geq 1} (k+1) y^k 
  =2y + 4 \frac{2y-y^2}{\big(1-y\big)^2}\tag{III}\\
  &=10y + 12y^2 + 16y^3 + 20y^4 + 24y^5 +\dots\notag
\end{align*}

The operator $\nc$ maps the two central regions with two unpaired
diameters onto each other.  It maps the central region with a single
unpaired diameter, say red, to the central region that has this
diameter paired and all other diameters of red colour.

\subsection{The grand total.}

Combining Equations~\eqref{eq:1}--\eqref{eq:3} and applying
Proposition~\ref{prop:decomposition} we obtain the generating function for Ptolemy
diagrams of type~$D$:
\begin{align}
  \label{eq:total}\notag
  \P_D(y) &=
  y\P_A^\prime(y)\frac{1+12\P_A(y)-\P_A^2(y)-2\P_A^3(y)}{1-2\P_A(y)-\P_A^2(y)}\\
  &=y + 16y^2 + 82y^3 + 500y^4 + 3084y^5 + 19400y^6 +\dots\notag
\end{align}
It follows from the algebraicity of $\P_A(y)$ that also $\P_D(y)$ is
an algebraic generating function.  However, the equation is not
particularly appealing:
\begin{multline*}
  (4y^2 - 47y^2 - 48y + 8)\P^3(y) \\
  - 2(y - 2)(4y^3 - 47y^2 - 48y + 8)\P^2(y)\\
  +(4y^5 - 99y^4 + 628y^3 - 246y^2 - 240y + 32)\P(y)\\
  - 2y(20y^3 - 319y^2 + 152y + 16)=0
\end{multline*}


\end{document}